\documentclass[graybox]{svmult}

\usepackage{helvet}         
\usepackage{courier}        
\usepackage{type1cm}        
\usepackage{makeidx}         
\usepackage{graphicx}        
\usepackage{multicol}        
\usepackage[bottom]{footmisc}
\usepackage{bm}
\usepackage{mathtools}

\makeindex             

\usepackage{latexsym}
\usepackage{dsfont}
\usepackage{bbm}
\usepackage{amssymb}
\usepackage{amsmath}
\usepackage{enumitem}
\renewcommand{\geq}{\geqslant}
\renewcommand{\leq}{\leqslant}
\renewcommand{\ge}{\geqslant}
\renewcommand{\le}{\leqslant}

\newtheorem{Theorem}{Theorem}[section]
\newtheorem{Lemma}[Theorem]{Lemma}
\newtheorem{Cor}[Theorem]{Corollary}
\newtheorem{Prop}[Theorem]{Proposition}
\newtheorem{Rem}[Theorem]{Remark}

\def\cA{\mathcal{A}}
\def\cB{\mathcal{B}}
\def\cC{\mathcal{C}}

\def\cS{\mathcal{S}}

\def\cU{\mathcal{U}}
\def\cV{\mathcal{V}}

\def\Erw{\mathbb{E}}

\def\N{\mathbb{N}}
  
\def\Prob{\mathbb{P}} 

\def\R{\mathbb{R}}

\def\Z{\mathbb{Z}}

\def\bC{\textit{\bfseries C}}

\def\vth{\vartheta}

\def\1{\vec{1}}
\def\3{{\ss}}

\def\eqdist{\stackrel{d}{=}}

\def\ovl{\overline}

\def\tp{{\widehat p}}
\def\tq{{\widehat q}}
\def\tX{{\widehat X}}
\def\tY{{\widehat Y}}
\def\qed{\hfill$\square$}
\def\lpa{\left(}
\def\rpa{\right)}
\def\lva{\left\lvert}
\def\rva{\right\rvert}
\def\lacc{\left\{}
\def\racc{\right\}}
\def\lcr{\left[}
\def\rcr{\right]}

\def\hJ{{\widehat J}}
\def\wJ{{\widetilde J}}
\def\hT{{\widehat T}}
\def\wT{{\widetilde T}}

\def\ptha{\partial_{\theta}}
\def\e{{\rm e}}

\def\wtil{\widetilde}
\def\fps{\equiv}

\usepackage[bookmarks=true,bookmarksopen=true,colorlinks=true,linkcolor=blue,citecolor=blue,urlcolor=blue]{hyperref}

\usepackage{xcolor}
\usepackage{csquotes}

\allowdisplaybreaks[4] 

\begin{document}

\title*{Persistence for a class of order-one autoregressive processes and Mallows-Riordan polynomials}
\titlerunning{Persistence for a class of AR(1) processes}
\author{Gerold Alsmeyer, Alin Bostan, Kilian Raschel and Thomas Simon}
\authorrunning{G.~Alsmeyer, A.~Bostan, K.~Raschel and T.~Simon}
\institute{Gerold Alsmeyer \at Institute of Mathematical Stochastics, Department
of Mathematics and Computer Science, University of M\"unster,
Orl\'eans-Ring 10, 48149 M\"unster, Germany.
\email{gerolda@math.uni-muenster.de} \\
Alin Bostan \at Inria, Universit\'e Paris-Saclay, 1 rue Honor\'e d'Estienne d'Orves, 91120 Palaiseau, France.
\email{alin.bostan@inria.fr}
\\Kilian Raschel \at CNRS, Laboratoire Angevin de Recherche en Math\'ematiques, Universit\'e d'Angers, 2 boulevard Lavoisier, 49045 Angers, France. \email{raschel@math.cnrs.fr}
\\Thomas Simon \at Laboratoire Paul Painlev\'e,
Universit\'e de Lille, Cit\'e Scientifique, 
59655 Villeneuve d'Ascq, France. \email{thomas.simon@univ-lille.fr}
\\\\
{This project was partially funded by the Deutsche Forschungsgemeinschaft (DFG) under Germany's Excellence Strategy EXC 2044--390685587 (Gerold Alsmeyer) and by the European Research Council (ERC) under the European Union's Horizon 2020 research and innovation programme under the Grant Agreement No 759702 (Kilian Raschel). Alin Bostan and Kilian Raschel were also supported in part by DeRerumNatura ANR-19-CE40-0018.}}

\maketitle

\date{\today}

\abstract{We establish exact formulae for the persistence probabilities of an AR(1) sequence with symmetric uniform innovations in terms of certain families of polynomials, most notably a family introduced by Mallows and Riordan as enumerators of finite labeled trees when ordered by inversions. The connection of these polynomials with the volumes of certain polytopes is also discussed. Two further results provide factorizations of general AR(1) models, one for negative drifts with continuous innovations, and one for positive drifts with continuous and symmetric innovations. The second factorization extends a classical universal formula of Sparre Andersen for symmetric random walks. Our results also lead to explicit asymptotic estimates for the persistence probabilities.}
\bigskip



\section{Introduction and main results}

\label{sec:introduction}

Let $X$ be a real Gaussian random variable with mean $\mu$ and variance $\sigma^{2}$. Then $Y=e^X$ has a log-normal distribution with integer moments $\Erw [Y^{n}] =e^{n\mu+n^{2}\sigma^{2}/2}$ for all $n\ge 1$ and cumulant generating function
\begin{equation}
\label{LogNom}
\log \Erw [e^{tY}]\; \fps\; \log\lpa\sum_{n\geq 0} e^{n\mu+n^{2}\sigma^{2}/2}\,\frac{t^{n}}{n!}\rpa\; \fps\; \log\lpa\sum_{n\geq0}\theta^{n(n-1)/2}\,\frac{z^{n}}{n!}\rpa,
\end{equation}
where $\theta=e^{\sigma^{2}}$ and $z=te^{\mu+\sigma^{2}/2}$. However, being divergent except for the degenerate case $\sigma^{2}=0$, the above series are considered as formal power series only, and the notation $\fps$ is used here and throughout to express identity between two such series. 

It was observed by Mallows and Riordan in \cite{MaRi68} that the right-hand side of \eqref{LogNom} is the exponential generating function of a family of polynomials with integer coefficients. More precisely, Eq.\,(2) in \cite{MaRi68} asserts that
\begin{equation}
\label{id:MR}
\log \left( \sum_{n\ge 0} \theta^{n(n-1)/2}\, \frac{z^{n}}{n!}\right)\; \fps \;\sum_{n\ge 1}\, (\theta-1)^{n-1} J_{n}(\theta)\, \frac{z^{n}}{n!}
\end{equation}
where $J_{n}(\theta)\in \Z[X]$ is a polynomial of degree $(n-1)(n-2)/2$ with positive coefficients and leading coefficient 1. For $n=1,\ldots,6$, one finds that
\begin{align*}
   J_{1}(\theta) &\; =\; 1,\\
   J_{2}(\theta) &\; =\; 1,\\
   J_{3}(\theta) &\; =\; 2\, +\,\theta,\\ 
   J_{4}(\theta) &\; =\; 6\, +\, 6\theta\, +\, 3\theta^{2}\, +\, \theta^3,\\
   J_{5}(\theta) &\; =\; 24\, +\, 36\theta\, +\, 30\theta^{2}\, +\,20\theta^3\,+\,10\theta^4\,+4\theta^5\,+\,\theta^6,\\
   J_{6}(\theta) &\; = \; 120\, +\, 240\theta\, +\, 270\theta^{2}\, +\, 240\theta^3\, +\, 180\,\theta^4\,+\,120\theta^{5}\, +\, 70\theta^{6} \\ 
& \qquad\qquad\qquad\quad\, +\, 35\theta^7\,+\, 15\theta^8\, +\,5\theta^9\, +\, \theta^{10}.
\end{align*} 
The main result of \cite{MaRi68} is that, for each $n\ge 1$, $J_{n}(\theta)$ equals the enumerator of trees with $n$ labeled points by number of inversions, when inversions are counted on each branch and ordered away from the root, which receives label 1. In particular, one has $J_{n}(1) = n^{n-2}$ for all $n\ge 1$ by Cayley's formula. This combinatorial significance leads to a recursive formula for these polynomials, namely 
\begin{equation}
\label{recMR}
J_{n+2}(\theta)\; =\; \sum_{i=0}^{n} \binom{n}{i} (1 + \theta + \cdots + \theta^{i})\, J_{i+1}(\theta)\, J_{n+1-i} (\theta)
\end{equation}
for every $n\ge 0$, see Formula (1) in \cite{Kreweras80}. Equivalently, one has   
\begin{equation}
\label{id:Mathar}
\sum_{n\ge 0} J_{n+1}(\theta)\, \frac{z^{n}}{n!} \; \fps\; \exp\lcr \sum_{n\ge 1} J_{n}(\theta)\, (1+\theta+\cdots + \theta^{n-1})\, \frac{z^{n}}{n!}\rcr
\end{equation}
by Cauchy's product, integration and identifying coefficients, see Formula (5) in \cite{MaRi68}. Mallows-Riordan polynomials, sometimes also called inversion polynomials in the literature, appear in many other counting problems, see \cite{FSS04} and the references therein. For example, it was shown in \cite{GW79} that $\theta^{n-1} J_{n}(\theta +1)$ is the enumerator of connected labeled $n$-vertex graphs by number of edges, thus $J_{n}(2)$ is the total number of these graphs. Mallows-Riordan polynomials are also an instance of Tutte's bivariate dichromatic polynomials in \cite{Tutte67} with one variable fixed, a topic we will discuss in some more detail in Paragraph~\ref{SoliTutti}. Let us finally mention that several conjectures on further combinatorial aspects of Mallows-Riordan polynomials are stated by Sokal in \cite{Sokal14, Sokal09}.\\

The main purpose of this paper is to provide a probabilistic interpretation of Mallows-Riordan polynomials that is not only quite different from the above connection with the log-normal distribution, but in fact also rather unexpected. To be more specific, consider a real autoregressive sequence of order one, defined by
\begin{equation}
\label{eq:AR(1)}
    Y_{0}\,=\,0\qquad\text{and}\qquad Y_{n}\ =\ \theta Y_{n-1}\,+\,X_{n}\quad\text{for }n\ge 1, 
\end{equation}
with drift parameter $\theta\in\R$ and i.i.d.~innovations $X_{1},X_{2},\ldots$ with a nondegenerate law. The ergodic properties of this Markov chain with continuous state space, which can be viewed as a discrete version of the Ornstein-Uhlenbeck process, are well known. It follows from \cite[Corollary 4.3]{GolMal:00}, see also \cite[Proposition 1]{GZ04}, that it is positive recurrent if and only if  $\theta\in (-1,1)$ and $\Erw \log (1 + \vert X_{1}\vert) < \infty$. In this case, the chain converges in distribution towards its unique stationary regime, given by the law of the so-called perpetuity
\begin{equation*}
   Y_\infty\; =\; \sum_{n\ge 1}\,\theta^{n-1} X_{n}.
\end{equation*}
Let $T_{\theta} = \inf\{n\ge 1:Y_{n} < 0\}$ be the time where the process becomes negative and consider the corresponding persistence probabilities 
\begin{equation*}
   p_{n}(\theta)\;=\; \Prob [T_{\theta} > n] \; =\; \Prob [Y_{1}\ge 0,\ldots,Y_{n} \ge 0], \quad n\ge 1.
\end{equation*}
The asymptotic behaviour of $p_{n}(\theta)$ as $n\to\infty$ has been recently investigated in \cite{AMZ21, DDY, HKW20} within a broader class of autoregressive models. See also the end of Section 3 in \cite{AS15} and the references therein for a heuristic discussion. In the positive recurrent case, and for a two-sided innovation law, the results in \cite{AMZ21} provide conditions for the rough estimate

\begin{equation}
\label{AMX}
   p_{n}(\theta) \; =\; \lambda^{n+ o(n)},
\end{equation}
where $\lambda\in (0,1)$ denotes the largest eigenvalue of some associated compact operator whose explicit form is typically unknown. See also \cite{AK19} for an asymptotic study in the case of Gaussian innovations. In the case $\theta\in (0,1)$ and for a two-sided innovation law with bounded support and positive absolutely continuous component, the results in \cite{HKW20} show that the rough estimate \eqref{AMX} can be refined to$p_{n}(\theta)\,\sim \,c\lambda^n$ for some positive constant $c.$ 

\vspace{.1cm}

The present work aims at providing some exact formulae for certain such order 1 autoregressive persistence probabilities, which then also lead to more explicit asymptotics. Our first main result establishes the announced unexpected relationship with Mallows-Riordan polynomials and considers the case when the innovation law is uniform on $[-1,1]$. For this case, we stipulate that $p_{n}^{U}(\theta)$ and $T_{\theta}^{\,U}$ are used hereafter for $p_{n}(\theta)$ and $T_{\theta}$, respectively.

\begin{theorem}
\label{thm:main_1}
For any $\theta\in[-1,\frac{1}{2}]$ and $n\ge 1$, one has
\begin{equation}\label{eq:MR-id}
p_{n}^{U}(\theta)\; =\; \frac{J_{n+1}(\theta)}{2^{n}\,n!}\cdot
\end{equation}
\end{theorem}
 
If $\theta = 0$, it follows directly that $p_{n}^{U}(0)=2^{-n}$, thus giving $J_{n+1} (0) = n!$ for any $n$, an immediate consequence also of Eq.\,\eqref{id:MR}. It can be interpreted combinatorially by the well-known bijection between permutations and ordered labeled trees obtained via the contour function. If $\theta = -1$, a formal differentiation of \eqref{id:MR} and some trigonometry shows that  
\begin{equation*}
   \sum_{n\ge 0} J_{n+1}(-1)\, \frac{z^{n}}{n!}\; = \;  \frac{1 + \sin z}{\cos z}
\end{equation*}
and then by comparison of coefficients that $J_{n+1}(-1) = A_{n}$, where $A_{n}$ denotes Euler's $n$-{th} zigzag number, see also Propri\'et\'e 2 in \cite{Kreweras80} for a derivation based on the recursion \eqref{recMR}. In Remark \ref{Zzz}(d) below, we provide a combinatorial explanation of this formula by relating $A_{n}$ to the probability $p_{n}^{U}(-1)$ when viewed as the renormalized volume of a certain polytope. In the upper boundary case $\theta = \frac{1}{2}$, it has been observed in \cite{GSY95, Rob73} that $2^{n(n-1)/2} J_{n+1}(\frac{1}{2})$ equals the number of initially connected acyclic digraphs with $n+1$ vertices, where ``initially connected'' means that there is a directed path from the vertex labeled 1 to any other vertex of the digraph. The relationship between this quantity and $p_{n}^{U}(\frac{1}{2})$ may appear even more surprising. 

\vspace{.1cm}
We present three proofs of Theorem \ref{thm:main_1}, all to be found in Subsection \ref{sec:proof_main_theorem}. The first one relies on a linear recurrence relation between the $p_{n}(\theta)$ that leads to a closed-form expression of their generating function as a ratio, see Proposition \ref{prop:recurrence_relation}, which in turn is of the same kind as a combinatorial formula for Mallows-Riordan polynomials stated in \cite[Eq.\,(14.6)]{Gessel80}. The other two proofs are  variations, the first one via multivariate changes of variable and the second one by using that $(\theta-1)^{n} J_{n+1} (\theta)/n!$ can be viewed as the algebraic volume of a certain polytope, see Proposition \ref{prop:MR_volume}. This volume interpretation is further discussed in Paragraph \ref{SoliTutti} in the framework of Tutte polytopes. Thanks to the exact character of Theorem \ref{thm:main_1}, we are able in Subsection \ref{Asyb} to derive some precise asymptotics for $p_{n}^{U}(\theta)$ and even a complete asymptotic expansion if $\theta\in [0,\frac{1}{2}]$. In particular, the exponential rate $\lambda$ in \eqref{AMX} is identified via the first negative root of the deformed exponential function
\begin{equation}\label{eq:deformed exp}
E(\theta, z) \; =\; \sum_{n\ge 0}\,\theta^{n(n-1)/2}\,\frac{z^{n}}{n!}\cdot
\end{equation}
Equivalently, this identification provides the spectral gap associated with a class of truncated Volterra operators, see Remark \ref{LastR}(c). Finally, some further infinite divisibility properties related to $T_{\theta}^{U}$ are discussed in Paragraph \ref{ID}.

\vspace{.2cm}
Eq.\,\eqref{eq:MR-id} established by Theorem~\ref{thm:main_1} works only for $\theta\in[-1,\frac{1}{2}]$. Namely, for $\theta$ outside this interval the situation changes because the domain of integration of the integral defining $p_{n}^{U}(\theta)$ undergoes truncations that make it behave differently as a function of $\theta$. Therefore, we will show by two further theorems, assuming $\theta<0$ and $\theta>0$, respectively, that a relation between the family $\{p_{n}(\theta):k=0,\dots n\}$ and its involutive conjugate $\{p_{n}(1/\theta):k=0,\ldots,n\}$ can be established for each $n\ge 1$. For $\theta<-1$ and $\theta\ge 2$, these relations can then be utilized to derive identites for $p_{n}(\theta)$ in terms of two new families of polynomials with integer coefficients, denoted $\wJ_{n}({\theta})$ and $\hJ_{n}({\theta})$, respectively, and still related to the Mallows-Riordan polynomials $J_{n}(\theta)$ through identities in terms of their generating functions, see Corollaries \ref{cor:Thm 2} and \ref{cor:Thm 3}.

\vspace{.1cm}
Here is the first of the two announced results, for which we go back to the initial model \eqref{eq:AR(1)} with negative drift $(\theta<0)$ and an innovation law whose non-negative part does not have atoms. 

\begin{theorem}
\label{thm:main_2}
Assuming $\theta<0$ and that the innovation law in \eqref{eq:AR(1)} has no  atoms on $[0,\infty),$ the relation
\begin{equation}\label{eq:duality theta<0}
\sum_{k=0}^{n} (-1)^k\, p_{k}(\theta)\,p_{n-k}(1/\theta)\; =\; 0
\end{equation}
holds for all $n\ge 1$.
\end{theorem}

The proof of this result, to be found in Subsection \ref{sec:case_<-1}, relies on a linear recurrence similar to the one derived in our first proof of Theorem~\ref{thm:main_1}. It further hinges on the fact that the $p_{n}(\theta)$ can be expressed in terms of the finite dual perpetuities $\sum_{k=1}^{n}r^{k-1}X_{k}$, where $r:=-1/\theta$. This is where the involution comes into play and our argument only requires the absence of non-negative atoms, a condition which also seems to be necessary by Remark \ref{Atomic} below.  

If $\theta =-1$ and the innovations are uniform on $[-1,1]$, then Theorem~\ref{thm:main_2} yields the well-known alternation property of Euler's zigzag numbers:
\begin{equation*}
   \sum_{k= 0}^{n} \binom{n}{k} (-1)^k A_{k}\, A_{n-k} \; = \; 0,\quad n\ge 1,
\end{equation*}
which corresponds to the trivial identity
\begin{equation*}
   \lpa\frac{1 - \sin z}{\cos z}\rpa\lpa\frac{1 +\sin z}{\cos z}\rpa\; =\; 1
\end{equation*}
between generating functions. 

\begin{corollary}\label{cor:Thm 2}
If $\theta\le -1$ and innovations are uniform on $[-1,1]$, then
\begin{equation}\label{eq:main_theorem_2}
p_{n}^{U}(\theta)\; =\; \frac{\wJ_{n+1}(1/\theta)}{2^{n}\, n!}
\end{equation}
for all $n\ge 0$, where the $\wJ_{n+1}({\theta})$ are defined by the identity
\begin{equation}\label{eq:modified_MR}
\sum_{n\ge 0} \wJ_{n+1}({\theta})\, \frac{z^{n}}{n!}\;\fps \lpa \sum_{n\ge 0} (-1)^{n} J_{n+1}(\theta)\, \frac{z^{n}}{n!}\rpa^{\!-1}
\end{equation}
and again a family  of polynomials in $\Z[X]$. Moreover, $\wJ_{1}(\theta):=1$ and $\wJ_{n+1}({\theta})$ has degree $n(n-1)/2$, valuation $n-1$ and coefficients of constant sign which alternates with $n$.
\end{corollary}

The $\wJ_{n}(\theta)$ for $n=2,\ldots,6$ are easily found explicitly with the help of \eqref{eq:modified_MR}, viz.:
\begin{align*}
   \wJ_2(\theta) &\; =\; 1,\\
   \wJ_3(\theta) &\; =\;  -\theta, \\ 
   \wJ_4(\theta) &\; =\;   3\theta^{2}\, +\, \theta^3, \\
   \wJ_5(\theta) &\; = \; -\lpa 12\theta^3\,+\,10\theta^4\,+4\theta^5\,+\,\theta^6\rpa,\\
   \wJ_6(\theta)&\; = \; 60\theta^4\,+\,80\theta^5\, +\, 60\theta^6\, +\, 35\theta^7 \, +\, 15\theta^8\,+\, 5\theta^9\, +\, \theta^{10}.
\end{align*}

Corollary \ref{cor:Thm 2} is proved in Subsection \ref{SpwJ}. At $\theta =-1$, the unique fixed point of the involution $\theta\mapsto 1/\theta$ on the negative halfline, Theorems \ref{thm:main_1} and \ref{thm:main_2} yield a remarkable order-$2$ phase transition for the mapping $\theta\mapsto p^{\,U}_{n}(\theta)$ for each $n\ge 2$. Namely, its first derivative is continuous at $-1$, but its second derivative is not, see Propositions~\ref{Phase1} and~\ref{Phase2}. As in the case $\theta\in [-1,\frac{1}{2}]$, the exact formula \eqref{eq:main_theorem_2} gives precise exponential asymptotics for $p^{\,U}_{n}(\theta)$  if $\theta < -1$, expressed in terms of the first negative root of the deformed exponential function $E(1/\theta,z)$, see Proposition \ref{Asymp2}.\\

For $\theta>0$, another relation between the $p_{n}(\theta)$ and their involutive duals $p_{n}(1/\theta)$ holds as shown by our next theorem, under the further assumption that the innovation law is symmetric and continuous. In combination with Eq.\,\eqref{eq:MR-id} for $\theta\in (-1,\frac{1}{2}]$, this further implies an identity for $p_{n}(\theta)$ in terms of a polynomial $\hJ_{n+1}(\theta)$ for all $n\ge 0$ and $\theta>2$.

\begin{theorem}\label{thm:main_3}
Assuming $\theta>0$ and the innovation law in \eqref{eq:AR(1)} to be continuous and symmetric, the relation
\begin{gather}\label{eq:duality theta>0}
\sum_{k=0}^{n}p_{k}(\theta)\,p_{n-k}(1/\theta)\; =\; 1
\intertext{holds for all $n\ge 0$ or, equivalently,}
\Bigg(\sum_{n\ge 0}\, p_{n}(\theta)\, z^{n}\Bigg)\Bigg(\sum_{n\ge 0}\, p_{n}(1/\theta) \,z^{n}\Bigg)\; = \; \frac{1}{1-z}\nonumber
\end{gather}
for all $z\in (-1,1).$
\end{theorem}

The proof of this result is given in Subsection \ref{sec:case_>1} and similar to the one of Theorem \ref{thm:main_2} by relying on a linear recurrence relation that is combined with a telescoping argument. Theorem \ref{thm:main_3} can also be viewed as an extension of the formula 
\begin{equation}\label{Sparre}
\sum_{n\ge 0}\, p_{n}(1) \,z^{n}\; =\; \frac{1}{\sqrt{1-z}}
\end{equation}
for every $z\in (-1,1)$ which goes back to Sparre Andersen in the case $\theta=1$, i.e., for ordinary random walks, see e.g.\ \cite[Eq.\,(2.3)]{AS15} and the references therein. Observe that the symmetry and the continuity of the increments are necessary for (\ref{Sparre}).

\begin{corollary}\label{cor:Thm 3}
If $\theta\ge 2$ and innovations are uniform on $[-1,1]$, then
\begin{equation}\label{eq:main_theorem_3}
p_{n}^{U}(\theta)\; =\; \frac{\hJ_{n+1}(1/\theta)}{2^{n}\, n!}
\end{equation}
for all $n\ge 0$, where the $\hJ_{n+1}({\theta})$ are defined by the identity
\begin{equation}\label{eq:modified_MR_bis}
\sum_{n\ge 0} \hJ_{n+1}({\theta})\, \frac{z^{n}}{n!}\; \fps\; \frac{1}{1-2z}
\Bigg(\sum_{n\ge 0} {J}_{n+1}({\theta})\,\frac{z^{n}}{n!}\Bigg)^{\!\!-1}
\end{equation}
and again polynomials in $\Z[X]$. Moreover, $\hJ_{1}(\theta):=1$ and $\hJ_{n+1}({\theta})$ is of degree $n(n-1)/2$ with order 0 coefficient $2^{n-1}n!$ and all other coefficients being negative.
\end{corollary}

For $n=2,\ldots,6$, the $\hJ_{n}(\theta)$ are as follows:
\begin{align*}
   \hJ_2(\theta) &\; =\; 1,\\
   \hJ_3(\theta) &\; =\;  4\,-\,\theta, \\ 
   \hJ_4(\theta) &\; =\;  24 \,-\,\lpa 6\theta\,+\,3\theta^{2}\,+\,\theta^3\rpa,\\
   \hJ_5(\theta) &\; = \; 192\, -\lpa 48\theta\, +\, 24\theta^{2}\, +\, 20\theta^3\,+\,10\theta^4\,+4\theta^5\,+\,\theta^6\rpa,\\
   \hJ_6(\theta)&\; = \;  1920\, -\, \lpa 480\theta\, +\, 240\theta^{2}\, +\, 200\theta^3\, +\, 160\,\theta^4\,+\,120\theta^5\, +\, 70\theta^6\right. \\ & \qquad\qquad\qquad\qquad \left. +\, 35\theta^7\,+\, 15\theta^8\, +\,5\theta^9\, +\, \theta^{10}\rpa.
\end{align*}

Corollary \ref{cor:Thm 3} is proved in Subsection \ref{SphJ} along with a further result on the first $k$ coefficients, for arbitrary $k\ge 0$, of the polynomial expansion of $p_{n}^{U}(\theta)$ for $n\ge k+1$ as a function of $1/\theta$. Namely, if $\theta\ge 2$, then these coefficients are independent of $n$ as asserted by Proposition \ref{MischMisch}. This in turn will lead to an explicit formula, see \eqref{Limes} in Subsection \ref{Asymp3}, for the all-time persistence probability 
\begin{equation*}
   \ell(\theta) \; = \; \lim_{n\to\infty} p^{\,U}_{n}(\theta)\; >\; 0,
\end{equation*}
and to a precise evaluation of the exponential rate of this convergence in terms of the first root of a certain meromorphic function, see Proposition~\ref{ExpQ}. In the more general case of continuous, symmetric innovations and for any $\theta>1$, we show in Proposition \ref{Hout} that the all-time persistence probability is also positive in the positive recurrent case. 

\vspace{.1cm}
For $\theta\in(\frac{1}{2},1)\cup (1,2)$, the quantities $p_{n}^{U}(\theta)$ appear as truncated Laurent series defined piecewise on a growing number of subintervals whose boundaries are generalized Fibonacci numbers and their inverses. The increasingly complicated formulae, which do not seem to have any combinatorial interpretation, are different on each Fibonacci subinterval of $(\frac{1}{2},1)$ or $(1,2)$; for two examples see Remarks \ref{MRGenLin}(a) and \ref{MRGL}(b). On the other hand, Figure~\ref{fig:plot_2345} shows the intriguing fact that the mapping $\theta\mapsto p_{n}^{U}(\theta)$ is apparently smoother at these boundaries than at the particular value $\theta = -1$ on the negative halfline.

\medskip

\begin{figure}
\begin{center}
\includegraphics[width=10cm]{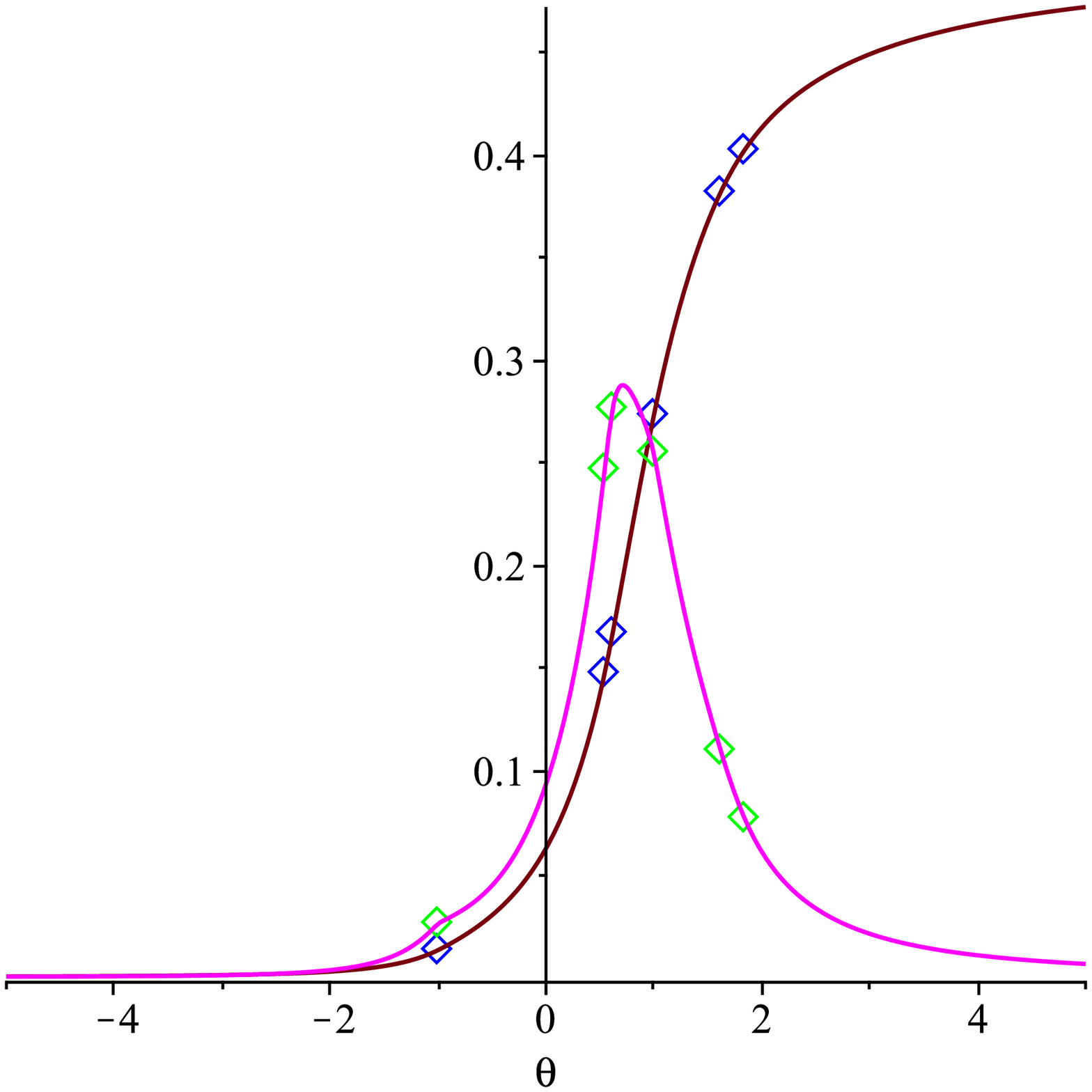}
\includegraphics[width=10cm]{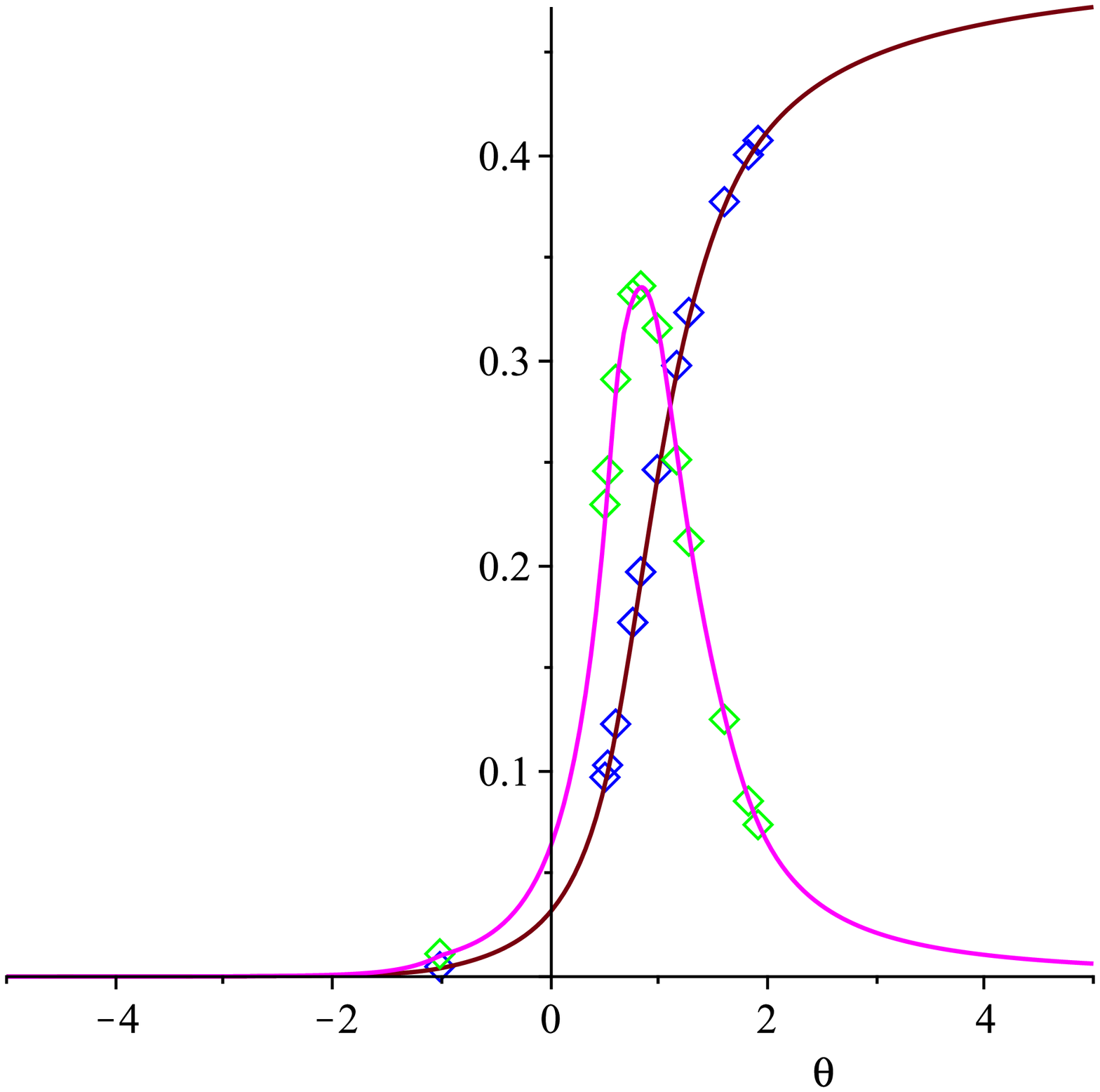}
\end{center}
\caption{The persistence probability $p_{n}(\theta)$ and its first derivative for $n=4,5$ and $\theta\in[-5,5]$. The blue points with positive abscissa indicate where the formula for $p_{n}(\theta)$ changes and correspond to the unique positive solutions to $\theta +\cdots +\theta^i = 1$ and to $1/\theta + \cdots + 1/\theta^i = 1$ for $i=1,\ldots, n-1$. At the blue point with negative abscissa $-1$, the first derivative of $p_{n}(\theta)$ is continuous, but the second derivative is not.}
\label{fig:plot_2345}
\end{figure}

The above discussion shows that for symmetric uniform innovations, the general and very simple factorization stated in Theorem~\ref{thm:main_3} for $\theta > 0$ leads to explicit expressions for the appearing coefficients in the case $\theta \in (0,\frac{1}{2}]$, owing to Theorem \ref{thm:main_1} and Corollary \ref{cor:Thm 3}, but does not in the remaining case $\theta \in (\frac{1}{2},1)$. More precisely, if there are explicit formulae, they will be much more complicated. In the case of biexponential symmetric innovations previously studied in \cite{Larr04} for $\theta\in (0,1)$, such a strong difference of complexity occurs as well, namely between the cases $\theta < 0$ and $\theta > 0$, see Remarks \ref{Expo} and \ref{MRGL}(a). As for the Wiener-Hopf factorization of random walks, it would be interesting to know if there are other AR(1) models with continuous symmetric innovations that have explicit factors in Theorems \ref{thm:main_2} and \ref{thm:main_3} given in terms of combinatorial objects as remarkable as the Mallows-Riordan polynomials.

\section{The case $\theta\in [-1,\frac{1}{2}]$}

\subsection{Proof of Theorem \ref{thm:main_1}}

\label{sec:proof_main_theorem}

By definition,
\begin{equation*}
   p_{n}^{U}(\theta)\ =\ \Prob[U_{1}>0,U_{2}+\theta U_{1}>0,\ldots,U_{n}+\theta U_{n-1}+\ldots+\theta^{n-1}U_{1}>0] 
\end{equation*}
for any $\theta\in\mathbb R$ and $n\ge 1$. For $\theta \ge 0$, this leads to the truncated integral formula
\begin{equation}
\label{Trunk}
p_{n}^{U}(\theta) \, =\, \frac{1}{2^{n}}\int_{0}^{1}\!\int_{-(1\wedge \theta u_{1})}^{1}\!\!\!\cdots\int_{-(1\wedge(\theta u_{n-1}+\cdots+\theta^{n-1}u_{1}))}^{1}du_{n}\ldots du_{2}\,du_{1}
\end{equation}
which remains valid for $\theta \in [-1,0)$ because $\theta u_{i-1}+\cdots+\theta^{i-1}u_{1}\in [-1,0]$ for each $i = 2,\ldots, n$ in the domain of integration of the $u_{i}$. For $\theta \in [-1,\frac{1}{2}]$, the crucial point is now, technically speaking, that $\theta u_{i-1}+\cdots+\theta^{i-1}u_{1}\in [-1,1]$ for every $i = 2,\ldots, n$ and thus all truncations in the domain of integration can be removed in (\ref{Trunk}), giving
\begin{equation}
\label{eq:why_theta<=1/2}
p_{n}^{U}(\theta) \ =\ \frac{1}{2^{n}}\int_{0}^{1}\!\int_{-\theta u_{1}}^{1}\!\!\!\cdots\int_{-(\theta u_{n-1}+\cdots+\theta^{n-1}u_{1})}^{1}du_{n}\ldots du_{2}\,du_{1},
\end{equation}
which will be our starting point. We mention that, if $\theta_{n}>\frac{1}{2}$ is the positive solution to $\theta + \cdots + \theta^{n-1} =1$, thus $\theta_{n}\downarrow\frac{1}{2}$ as $n\to\infty$, then \eqref{eq:why_theta<=1/2} actually remains true for $\theta \not\in [-1,\theta_{n}]$, but fails to be so for $\theta > \theta_{n}$, see Remark \ref{MRGenLin}(a) below. In the following, we put $\tp_{n}^{\,U}(\theta) = 2^{n}p_{n}^{U}(\theta)$.

\subsubsection{Proof via linear recurrence}
\label{subsec:proof_th1}

The following expression of the generating function of the  $\tp_{n}^{\,U}(\theta)$, as a ratio of two formal power series, forms the basis of the proof and is also of independent interest.

\begin{proposition}
\label{prop:recurrence_relation}
For every $\theta\in [-1,\frac{1}{2}]$, one has
\begin{equation}\label{eq:generating_recursion}
\sum_{n\ge 0} \tp_{n}^{\,U}(\theta)\, z^{n}\ \fps\ \frac{\displaystyle 1+\sum_{n\ge 2}\frac{(\theta+\cdots+\theta^{n-1})^{n}}{\theta^{n(n-1)/2}}\,\frac{z^{n}}{n!}}{\displaystyle 1-z+\sum_{n\ge 3}\frac{(\theta+\cdots+\theta^{n-2})^{n}}{\theta^{n(n-1)/2}}\,\frac{z^{n}}{n!}}\cdot
\end{equation}
\end{proposition}

\begin{proof}
We put $H_{n}(u)=u^{n}/n! = H_{n+1}'(u)$ for any $n\ge 0$ and rewrite \eqref{eq:why_theta<=1/2} as
\begin{equation*}
   \tp_{n}^{\,U}(\theta) \ =\ \int_{0}^{1}du_{1}\int_{-v_{1}}^{1}du_{2}\cdots\int_{-v_{n-1}}^{1}du_{n}, 
\end{equation*}
where $v_{0} = 0$ and $v_{i}=\theta (u_{i}+v_{i-1})$ for $i=1,\ldots, n-1$. Setting $\tq_{n}^{\,U}(\theta) = \tp^{\,U}_{n}(\theta)-\tp^{\,U}_{n-1}(\theta)$ for any $n\ge 2$ and $\tq_{1}(\theta) = 0$, we then see that 
\begin{align*}
\tq^{\,U}_{n}(\theta) \
&=\ \int_{-v_{0}}^{1}\int_{-v_{1}}^{1}\cdots\int_{-v_{n-2}}^{1}H_{1}(v_{n-1})\ du_{1}\cdots\,du_{n-1}\\
&=\ \frac{1}{\theta}\int_{-v_{0}}^{1}\int_{-v_{1}}^{1}\!\!\!\cdots\int_{-v_{n-3}}^{1}H_{2}(\theta+\theta v_{n-2})\ du_{1}\cdots\,du_{n-2}
\end{align*}
for each $n\ge 3$, and $\tq^{\,U}_2(\theta) = \theta/2$. Further partial integration provides
\begin{align*}
\tq^{\,U}_{n}(\theta)\
&=\ \frac{1}{\theta}\int_{-v_{0}}^{1}\int_{-v_{1}}^{1}\!\!\!\cdots\int_{-v_{n-3}}^{1}H_{2}(\theta+\theta^{2}(u_{n-2}+v_{n-3}))\ du_{1}\cdots\,du_{n-2}\\
&=\ \frac{1}{\theta^{3}}\int_{-v_{0}}^{1}\int_{-v_{1}}^{1}\!\!\!\cdots\int_{-v_{n-4}}^{1}\!\!\!\!\! (H_{3}(\theta+\theta^{2}(1+v_{n-3}))-H_{3}(\theta))\ du_{1}\cdots\,du_{n-3}\\
&=\ -\,\frac{H_{3}(\theta)\,\tp_{n-3}^{\,U}(\theta)}{\theta^{3}}\\
&\quad+ \ \frac{1}{\theta^{3}}\int_{-v_{0}}^{1}\int_{-v_{1}}^{1}\!\!\!\cdots\int_{-v_{n-4}}^{1}\!\!\!\!\!\! H_{3}(\theta+\theta^{2}+\theta^{3}(u_{n-3}+v_{n-4}))\ du_{1}\cdots\,du_{n-3}\\
&\;\;\vdots\hspace{3cm}\vdots\\
&=\ -\,\sum_{k=3}^{n-1}\frac{H_{k}(\theta+\cdots+\theta^{k-2})\,\tp_{n-k}^{\,U}(\theta)}{\theta^{k(k-1)/2}}\\
&\quad+\ \frac{1}{\theta^{(n-1)(n-2)/2}}\int_{0}^{1}H_{n-1}(\theta+\cdots+\theta^{n-2}+\theta^{n-1}u_{1})\ du_{1}\\
&=\ -\,\sum_{k=3}^{n}\frac{H_{k}(\theta+\cdots+\theta^{k-2})\,\tp_{n-k}^{\,U}(\theta)}{\theta^{k(k-1)/2}}\,+\,\frac{H_{n}(\theta+\cdots+\theta^{n-1})}{\theta^{n(n-1)/2}}.
\end{align*}
Consequently, $\tp_{1}^{\,U}(\theta)=\tp_{0}^{\,U}(\theta) =1,\,\tp_{2}^{\,U}(\theta)=\tp_{1}^{\,U}(\theta)+\theta/2$, and 
\begin{equation*}
   \tp_{n}^{\,U}(\theta)\; =\; \tp_{n-1}^{\,U}(\theta)\, +\,\frac{H_{n}(\theta+\cdots+\theta^{n-1})}{\theta^{n(n-1)/2}}\,-\,\sum_{k=3}^{n}\frac{H_{k}(\theta+\cdots+\theta^{k-2})\,\tp_{n-k}^{\,U}(\theta)}{\theta^{k(k-1)/2}} 
\end{equation*}
for any $n\ge 3$. Multiplication by $z^{n}$ and subsequent summation over $n$ finally leads to the following identity between formal power series:
\begin{align*}
(1-z)\,&\sum_{n\ge 0}\, \tp_{n}^{\,U}(\theta)\, z^{n}\\
&  \fps\ 1\ +\ \sum_{n\ge 2}\frac{H_{n}(\theta+\cdots+\theta^{n-1})\, z^{n}}{\theta^{n(n-1)/2}}\\
& \qquad\;\; -\ \sum_{n\ge 3}\sum_{k=3}^{n} \left(\frac{H_{k}(\theta+\cdots+\theta^{k-2})\, z^k}{\theta^{k(k-1)/2}}\right) \tp_{n-k}^{\,U} z^{n-k}\\
&  \fps\ 1\ +\ \sum_{n\ge 2}\frac{H_{n}(\theta+\cdots+\theta^{n-1})\, z^{n}}{\theta^{n(n-1)/2}}\\
& \qquad\;\; -\ \left(\sum_{n\ge 0} \tp_{n}^{\,U}(\theta)\, z^{n}\right)\left(\sum_{n\ge 3}\frac{H_{n}(\theta+\cdots+\theta^{n-2})\, z^{n}}{\theta^{n(n-1)/2}}\right),
\end{align*}
and this is easily seen to be equivalent to \eqref{eq:generating_recursion}. 
\qed
\end{proof}

\begin{proof}[of Theorem \ref{thm:main_1}] 
Differentiation with respect to $z$ of Eqs.\,\eqref{id:MR} and \eqref{eq:generating_recursion} provides
\begin{gather}
\label{eq:to_prove-MR}
\sum_{n\ge 0} (\theta-1)^{n} J_{n+1}(\theta)\frac{z^{n}}{n!}\ \fps\ \frac{\displaystyle\sum_{n\ge 0}\theta^{n(n+1)/2}\, \frac{z^{n}}{n!}}{\displaystyle\sum_{n\ge 0}\theta^{n(n-1)/2}\, \frac{ z^{n}}{n!}}
\shortintertext{and}
\label{eq:to_prove-0}
\sum_{n\ge 0} (\theta -1)^{n} \tp_{n}^{\,U}(\theta)\, z^{n}\ \fps\ \frac{\displaystyle \sum_{n\ge 0}\frac{(\theta^{n}-\theta)^{n}}{\theta^{n(n-1)/2}}\,\frac{z^{n}}{n!}}{\displaystyle \sum_{n\ge 0}\frac{(\theta^{n-1}-\theta)^{n}}{\theta^{n(n-1)/2}}\,\frac{z^{n}}{n!}},
\end{gather}
respectively, whence it is enough to show identity of the two ratios of formal power series on the right-hand sides. Equivalently,
\begin{multline*}
\Bigg(\sum_{n\ge 0}\frac{(\theta^{n}-\theta)^{n}}{\theta^{n(n-1)/2}}\,\frac{z^{n}}{n!}\Bigg)\Bigg(  \sum_{n\ge 0} \theta^{n(n-1)/2}\,\frac{z^{n}}{n!}\Bigg) \\ 
\qquad \fps \ \Bigg( \sum_{n\ge 0}\theta^{n(n+1)/2}\,\frac{z^{n}}{n!}\Bigg)\Bigg( \sum_{n\ge 0}\frac{(\theta^{n-1}-\theta)^{n}}{\theta^{n(n-1)/2}}\frac{z^{n}}{n!}\Bigg)
\end{multline*}
must be verified, that is, upon comparison of coefficients,
\begin{equation*}
   \sum_{k=0}^{n}\binom{n}{k} (\theta^{k-1} - \theta)^k\theta^{-n(k-1)}\;
= \; \sum_{k=0}^{n}\binom{n}{k}(\theta^{k} - \theta)^{k}\theta^{-k(n-1)} 
\end{equation*}
for each $n\ge 0$. To this end, we finally note that
\begin{align*}
\sum_{k=0}^{n}&\binom{n}{k} (\theta^{k-1}-\theta)^k\theta^{-n(k-1)}\\
&=\ \sum_{k=0}^{n}\binom{n}{k}\theta^{-n(k-1)}\sum_{\ell=0}^{k}\binom{k}{\ell} (-1)^{k-\ell}\theta^{k-\ell + \ell(k-1)}\\
&=\ \sum_{l=0}^{n}\sum_{k=\ell}^{n}\binom{n}{k}\binom{k}{\ell} (-1)^{k-\ell}\theta^{k-\ell + (\ell-n)(k-1)}\\
&=\ \sum_{\ell'=0}^{n}\sum_{k'=0}^{\ell'}\binom{n}{\ell'}\binom{\ell'}{k'}(-1)^{\ell' - k'}\theta^{\ell' - k' +\ell'(1+k'-n)}\\
&=\ \sum_{k=0}^{n}\binom{n}{k}\theta^{-k(n-1)}\sum_{\ell=0}^{k}\binom{k}{\ell} (-1)^{k-\ell}\theta^{k-\ell + \ell k}\\
&=\ \sum_{k=0}^{n}\binom{n}{k}(\theta^{k} - \theta)^{k}\theta^{-k(n-1)}
\end{align*}
where we have set $k' = n-k,\,\ell' = n-\ell$ in the fourth line and made the change of variable $(\ell', k')\mapsto (k,l)$ in the fifth line.\qed
\end{proof}

\begin{Rem} 
\label{MRGenLin}
{\em (a) It follows from the point made after \eqref{eq:why_theta<=1/2} at the beginning of this section that, for any fixed $n\ge 2$, the statement of Theorem \ref{thm:main_1} extends to all $\theta \in (\frac{1}{2},\theta_{n}]$ if $\theta_{n} > \frac{1}{2}$ denotes the positive solution to $\theta +\cdots + \theta^{n-1} = 1$, which has been coined generalized Fibonacci number in the literature, see \cite{W98} and the references therein. On the other hand, the formula becomes different for $\theta > \theta_{n}$. For example, for $n=3$ and $\theta\in (\theta_3,1)$, it can be shown that
\begin{equation*}
   \tp_3^{\,U}(\theta)\; =\; \theta \, +\,\frac{11}{6}\, -\, \frac{1}{2\theta^{2}}\, +\, \frac{1}{6\theta^3}\cdot
   \end{equation*}
Except for the Sparre Andersen formula at $\theta = 1$, expressing $\tp_{n}^{\,U}(\theta)$ as a function of $\theta$ for $\theta \in (\theta_{n},1/\theta_{n})$ becomes very complicated with growing $n$ and does not seem to permit a nice combinatorial formulation. An intriguing fact is that, despite these complications, the mapping $\theta\to\tp_{n}^{\,U}(\theta)$ seems to maintain a certain degree of smoothness for $\theta\in [\frac{1}{2},2]$ and any $n\ge 1$, see Figure~\ref{fig:plot_2345}.\\

(b) As a by-product of \eqref{eq:generating_recursion}, we obtain the following apparently new formula for the generating function of Mallows-Riordan polynomials as a ratio of two (for any $\theta\not\in\{-1,0,+1\}$ divergent) formal power series:
\begin{equation*}
   \sum_{n\ge 0} J_{n+1}(\theta) \frac{z^{n}}{n!}\ \fps \ \frac{{\displaystyle 1 + \sum_{n \ge 2} \frac{(\theta + \cdots + \theta^{n-1})^{n}}{\theta^{n(n-1)/2}}\,\frac{z^{n}}{n!}}}{{\displaystyle 1 - z + \sum_{n \ge 3} \frac{(\theta + \cdots + \theta^{n-2})^{n}}{\theta^{n(n-1)/2}}\,\frac{z^{n}}{n!}}}. 
   \end{equation*}
The formula corresponds to the linear recurrence
\begin{align*}
J_{n+1}(\theta)\ = \ &  \frac{(\theta+\ldots+\theta^{n-1})^{n}}{\theta^{n(n-1)/2}} \, + \, n J_{n}(\theta)\\
&\qquad\quad -\,\sum_{k=3}^{n} \binom{n}{k}\frac{(\theta+\ldots+\theta^{k-2})^k}{\theta^{k(k-1)/2}}\,J_{n-k+1}(\theta),
\end{align*}
valid for any $n\ge 2$, of $J_{n+1}(\theta)$ in terms of $J_{k}(\theta)$ for $k\le n$, with the notable curiosity that $J_{n-1}(\theta)$ does not appear. These formulae should be compared with the following identity due to Gessel, see Formula (14.6) in \cite{Gessel80}:
\begin{equation*}
   \sum_{n\ge 0} J_{n+1}(\theta) \,\frac{z^{n}}{n!}\; \fps \;  \frac{{\displaystyle \sum_{n \ge 0} \frac{(1+\theta + \cdots + \theta^{n})^{n}}{\theta^{n(n+1)/2}}\,\frac{z^{n}}{n!}}}{{\displaystyle \sum_{n \ge 0} \frac{(1+\theta + \cdots + \theta^{n-1})^{n}}{\theta^{n(n+1)/2}}\,\frac{z^{n}}{n!}}} 
   \end{equation*}
and the corresponding linear recurrence
\begin{equation*}
   \sum_{k=0}^{n}\binom{n}{k} \frac{(\theta-1)^{k}J_{k+1}(\theta)}{\theta^{k(k-1)/2-kn}}
\frac{(\theta^{n-k}-1)^{n-k}}{(\theta^{n+1}-1)^{n}}
\; =\; 1.
\end{equation*}
}
\end{Rem}

\subsubsection{Proof via polytope volumes}
\label{subsec:proof_th1_b}

This proof relies on the following expression of Mallows-Riordan polynomials as algebraic volumes, which is also of independent interest, see Section \ref{SoliTutti} below.

\begin{proposition}
\label{prop:MR_volume}
For every $n\ge 1$, one has the polynomial identity
\begin{equation}\label{eq:Alin's conj}
  \int_{1}^{\theta}\int_{1}^{\theta x_{1}}\cdots\;\int_{1}^{\theta x_{n-1}}dx_{1}\ldots dx_{n}  \ =\ \frac{(\theta-1)^{n} J_{n+1}(\theta)}{n!}\cdot 
\end{equation}
\end{proposition}

\begin{proof}
A homothetic change of variable provides
\begin{gather*}
\int_{1}^{\theta}\int_{1}^{\theta x_{1}}\cdots\;\int_{1}^{\theta x_{n-1}}dx_{1}\cdots dx_{n} \; = \; \theta^{n(n+1)/2} q_{n}(\theta),
\shortintertext{where}
q_{n}(\theta)\; =\; \int_{1/\theta}^{1}\int_{1/\theta^{2}}^{u_{1}}\cdots\;
\int_{1/\theta^{n}}^{u_{n-1}}du_{1}\ldots du_{n}
\end{gather*}
for $n\ge 1$. Next observe that
\begin{equation*}
   q_{n}(\theta) \; +\; \frac{q_{n-1}(\theta)}{\theta^{n}}\; =\; \int_{1/\theta}^{1}\int_{1/\theta^{2}}^{u_{1}}\cdots\;\int_{1/\theta^{n-1}}^{u_{n-2}} u_{n-1}\, du_{1}\ldots du_{n-1}
   \end{equation*}
and then, upon repeated summation,
\begin{equation*}
   \sum_{k=1}^{n}\frac{q_{k}(\theta)}{(n-k)!\,\theta^{(k+1)(n-k)}}\; = \; \frac{1}{(n-1)!} \int_{1/\theta}^{1} u_{1}^{n-1} du_{1}\; =\; \frac{1}{n!}\left(1-\frac{1}{\theta^{n}}\right). 
   \end{equation*}
Setting $\tq_{0}(\theta) = 1$ and 
\begin{equation*}
   \tq_{k}(\theta) \; =\; \frac{k!\, \theta^{k(k+1)/2} q_{k}(\theta)}{(\theta - 1)^k} 
   \end{equation*}
for $k\ge 1$, we obtain
\begin{equation}\label{eq:id_to_compare}
\sum_{k=0}^{n}\binom{n}{k}\frac{(\theta-1)^{k}\tq_{k}(\theta)}{\theta^{(k+1)(n-k/2)}} \ =\ 1
\end{equation}
for any $n\ge 0$. On the other hand, it follows from \eqref{eq:to_prove-MR} that
\begin{equation*}
   \lpa \sum_{n\ge 0} (\theta-1)^{n} J_{n+1}(\theta)\frac{z^{n}}{n!}\rpa \lpa \sum_{n\ge 0}\theta^{n(n-1)/2}\,\frac{z^{n}}{n!}\rpa\; \fps \; \sum_{n \ge 0}\theta^{n(n+1)/2}\, \frac{z^{n}}{n!}
   \end{equation*}
and therefore, by using Cauchy's product and comparing coefficients,
\begin{equation*}
   \sum_{k=0}^{n}\binom{n}{k}\frac{(\theta-1)^{k}J_{k+1}(\theta)}{\theta^{(k+1)(n-k/2)}} \ =\ 1 
   \end{equation*}
for any $n\ge 0$. As $\tq_{0}(\theta) = J_{1}(\theta) = 1$, we finally infer $\tq_{n}(\theta) = J_{n+1}(\theta)$ for each $n\ge 0$ from \eqref{eq:id_to_compare}, which completes the argument.\qed
\end{proof}

\begin{proof}[of Theorem \ref{thm:main_1}]
Embarking on \eqref{eq:why_theta<=1/2}, put $r = -1/\theta$ and make the changes of variables $y_{i+1}=x_{i+1} -\theta x_{i}$ (second line) and $z_{i+1} = r^i y_{i+1}$ (third line), both for $i=1,\ldots, n-1$, to obtain
\begin{align*}
\tp_{n}^{\,U}(\theta)\ &=\ \int_{0}^{1}\int_{-\theta x_{1}}^{1}\cdots\;\int_{-(\theta x_{n-1}+\cdots+\theta^{n-1}x_{1})}^{1}dx_{1}\ldots dx_{n}\\
&=\ \int_{0}^{1}\int_{0}^{1+\theta y_{1}}\cdots\;\int_{0}^{1+\theta y_{n-1}}\!\!\!dy_{1}\ldots dy_{n}\\
&=\ r^{-n(n-1)/2} \int_{0}^1 \int_{0}^{r-z_{1}} \cdots\; \int_{0}^{r^{n-1}-z_{n-1}}\!\!\! dz_{1}\ldots dz_{n}.
\end{align*}
By yet another change of variables, namely
\begin{equation*}
   t_{i}\; = \; 1\, -\, \frac{(r+1)z_{i}}{r^i} 
   \end{equation*} 
for $i=1,\ldots,n$, we arrive at 
\begin{equation*}
   \tp_{n}^{\,U}(\theta) \; = \; \lpa\frac{-r}{r+1}\rpa^{n} \int_{1}^{-1/r} \int_{1}^{-t_{1}/r}\cdots\;  \int_{1}^{-t_{n-1}/r}\!\! dt_{1}\ldots dt_{n} \; = \; \frac{J_{n+1}(\theta)}{n!} 
   \end{equation*}
as required, where Proposition \ref{prop:MR_volume} has been used for the last equality.
\qed
\end{proof}

\begin{Rem}\label{Inc1}
\rm It follows from the formula 
\begin{equation*}
   \tp_{n}^{\,U}(\theta) \; = \;\int_{0}^{1}\int_{0}^{1+\theta y_{1}}\cdots\;\int_{0}^{1+\theta y_{n-1}}\!\!\!dy_{1}\ldots dy_{n}
   \end{equation*}
that $\theta\mapsto p_{n}^{U}(\theta)$ is increasing on $[-1,\frac{1}{2}]$.
Therefore, by Theorem~\ref{thm:main_1} and since $J_{n}$ has positive coefficients, the mapping
\begin{equation*}
   \theta\; \mapsto \; J_{n}(\theta) 
   \end{equation*}  
is positive and increasing on $[-1,\infty)$ for all $n\ge 1$. We further note that the mapping $\theta\mapsto  p_{n}(\theta)$ is nondecreasing on $\R^{+}=[0,\infty)$ for any innovation sequence $\{X_{i},\,i\ge 1\}$ because $p_{n}(\theta)=\Prob [\Omega_{n}(\theta)]$ and
\begin{equation*}
   \Omega_{n}(\theta)\; :=\ \lacc \sum_{k=1}^j \theta^{j-k} X_{k} \ge 0, \; j=1,\ldots, n\racc, 
   \end{equation*}
is always nondecreasing in $\theta\ge 0$. However, this simple argument breaks down for $\theta < 0$. For Mallows-Riordan polynomials, a comparison of coefficients in the exponential 
formula \eqref{id:Mathar} plus an induction argument easily show that $J_{n+1}(\theta)$ is positive on $[-1,0]$. On the other hand, it seems impossible to show directly that $J_{n}'(\theta)\ge 0$ on $[-1,0)$ for all $n\ge 1$. For $\theta < -1$, we believe but were not able to verify that there exist some $n_{1}, n_2 \ge 1$ such that $J_{n_{1}}(\theta) < 0$ and $J_{n_2}'(\theta) < 0$. Finally, we mention that it has been conjectured in \cite{Sokal09} that all complex roots of $J_{n}$ lie outside the closed unit disk.
\end{Rem}
 
\subsubsection{Proof via bivariate generating functions}
\label{sec:Jan15}

Our last proof is particularly useful in the case $\theta = -1$ and embarks on the formula
\begin{align*}
\tp_{n}^{\,U}(\theta)\ =\ r^{-n(n-1)/2} \int_{0}^1 \int_{0}^{r-z_{1}} \cdots\; \int_{0}^{r^{n-1}-z_{n-1}}\!\!\! dz_{1}\ldots dz_{n}
\end{align*}
where $r = -1/\theta$ should be recalled. Suppose first that $\theta = -1$ and thus $r=1$. Setting $\cB_{0}(t)=1$ and   
\begin{equation}\label{eq:cB_{n}}
\cB_{n}(t)\; = \;\int_{0}^t \int_{0}^{1-z_{1}} \cdots\; \int_{0}^{1-z_{n-1}} dz_{1}\ldots dz_{n}
\end{equation}
for $n\ge 1$ and $t \ge 0$, we see that $\cB_{n}'(t)=\cB_{n-1}(1-t)$ and $\cB_{n}(0)=0$ for all $n\ge 1$. Introducing the bivariate generating function 
\begin{equation*}
   \cB(t,z)\; =\; \sum_{n \ge 0} \cB_{n}(t)\, z^{n}, 
   \end{equation*}
the ODE
\begin{equation}\label{eq:EDO_B}
\frac{\partial^{2} \cB}{\partial t^{2}}(t,z)\; =\; -z\,\frac{\partial \cB}{\partial t}(1-t,z)\; =\; -z^{2} \,\cB(t,z)
\end{equation}
holds as one can easily check, and it follows upon integration that 
\begin{equation*}
   \cB(t,z)\ =\ \cU(z)\, \cos tz + \cV(z)\, \sin tz. 
\end{equation*}
Moreover, use $\cB(0,z)=1$ and $\frac{\partial \cB}{\partial t}(1,z)=z$ to infer
\begin{gather*}
\cU(z)=1\quad\text{and}\quad\cV(z)\,=\,\frac{1+\sin z}{\cos z}\,\sin tz\,=\,\big(\sec z+\tan z\big)\sin tz,
\shortintertext{and thereby}
\cB(t,z)\,=\,\cos tz + \big(\sec z + \tan z\big)\,\sin tz.
\end{gather*}
Finally, we obtain upon setting $t=1$ that 
\begin{equation*}
   \sum_{n\ge 0} \tp_{n}^{\,U}(-1)\, z^{n}\; =\;  \cB_{n}(1,z)\; = \; \sec z + \tan z\; = \; 
\sum_{n\ge 0} A_{n}\, \frac{z^{n}}{n!} 
\end{equation*}
where $A_{n} = J_{n+1}(-1)$ denotes Euler's $n$-{th} zigzag number, see e.g.\ Propri\'et\'e~2 in \cite{Kreweras80} for the last identity. This completes the proof for $\theta=-1$.

\vspace{.2cm}
For the general case we define, as an extension of \eqref{eq:cB_{n}}, 
\begin{equation*}
   \cB_{n}(t,\theta)\; =\; \int_{0}^t \int_{0}^{r-z_{1}} \cdots\; \int_{0}^{r^{n-1}-z_{n-1}} dz_{1}\ldots dz_{n}
\end{equation*}
with $r=-1/\theta$ as before. Then $\cB_{n}(0, \theta) = 0$ and
\begin{equation*}
   \frac{\partial\cB_{n}}{\partial t}(t,\theta)\; =\; r^{n-1} \cB_{n-1}(1+\theta t,\theta). 
\end{equation*}
Integration and a change of variables provides
\begin{equation*}
   \cB_{n}(t,\theta) \; = \; r^{n-1} \int_{0}^t \cB_{n-1}(1+\theta s,\theta) \, ds\; = \; \frac{r^{n-1}}{\theta}  \int_{1}^{1+\theta t} \cB_{n-1}(u,\theta) \, du 
\end{equation*}
and then after $n-1$ iterations
\begin{align*}\cB_{n}(1,\theta) &\;=\;\frac{r^{n(n-1)/2}}{\theta^{n}} \int_{1}^{1+\theta}\int_{1}^{1+\theta x_{1}}\cdots\;\int_{1}^{1+\theta x_{n-1}}dx_{1}\cdots dx_{n} \\
&\;=\;\frac{r^{n(n-1)/2}}{(\theta -1)^{n}} \int_{1}^{\theta}\int_{1}^{\theta y_{1}}\cdots\;\int_{1}^{\theta y_{n-1}}dy_{1}\cdots dy_{n}\\
& \;=\; \frac{r^{n(n-1)/2} J_{n+1}(\theta)}{n!}
\end{align*}
by Proposition \ref{prop:MR_volume}. For the second equality we have made the multiple change of variables $x_{i}= (y_{i}+r)/(1+r)$ for $i = 1,\ldots, n$. As $\cB_{n}(1,\theta) = r^{n(n-1)/2}\tp_{n}^{\,U}(\theta)$, the proof is complete. 
\qed

\begin{Rem} \label{Zzz}
\rm (a) Regarding the proof just given for the case $\theta\ne-1$, it can be seen as a variation of the one in the preceding subsection, because it has been finalized by a use of Proposition \ref{prop:MR_volume}. Let us also note that, when defining the trivariate generating function 
\begin{equation*}
   \cB(t,\theta,z)\; =\; \sum_{n \ge 0} \cB_{n}(t,\theta)\, z^{n}, 
\end{equation*}
the ODE \eqref{eq:EDO_B} turns into the delayed PDE
\begin{equation*}
   \frac{\partial^{2} \cB}{\partial t^{2}}(t,\theta,z)\; = \;\theta z^{2} \,\cB(\theta^{2}t\theta+1,\theta,r^{2}z), 
\end{equation*}
which does not seem solvable in a simple manner.

\vspace{.2cm}
(b) An alternative way to finalize the above proof for $\theta\ne -1$ is via the volume of another polytope: By making successive changes of variables $x_{i}= 1+\theta y_{i}, y_{i}= \theta^{i-1} z_{i}$ and $z_{i}= \sum_{j=1}^{i-1} (\theta^{1-j} - u_j)$, we obtain
\begin{align*}
\cB_{n}(1,\theta)\ &=\ \frac{r^{n(n-1)/2}}{\theta^{n}} \int_{1}^{1+\theta}\int_{1}^{1+\theta x_{1}}\cdots\;\int_{1}^{1+\theta x_{n-1}}dx_{1}\cdots dx_{n} \\ 
&=\  r^{n(n-1)/2}\,\int_{0}^{1}\int_{0}^{1+\theta y_{1}}\cdots\;\int_{0}^{1+\theta y_{n-1}}dy_{1}\cdots dy_{n}\\
&=\  (-1)^{n(n-1)/2}\,\int_{0}^{1}\int_{0}^{\theta^{-1} +z_{1}}\cdots\;\int_{0}^{\theta^{1-n} +z_{n-1}}dz_{1}\cdots dz_{n}\\
&=\  (-1)^{n(n-1)/2}\,V_{n}(1,\theta^{-1},\ldots, \theta^{1-n})
\end{align*}
where $V_{n}(x_{1},\ldots, x_{n})$ denotes the volume of the polytope
\begin{equation*}
   \{ y_{i}\ge 0, \;  y_{1}+\cdots +y_{i}\le x_{1}+\cdots +x_{i},\; 1\le i\le n \}\ \subset\ \R^{n}.
\end{equation*}
Applying the formula in \cite[p.\,620]{StPi02} finally provides the required identity
\begin{equation*}
    \cB_{n}(1,\theta) \; =\; \frac{r^{n(n-1)/2} J_{n+1}(\theta)}{n!}\cdot 
\end{equation*}
Having this pointed out, Proposition \ref{prop:MR_volume} can be viewed as an elementary proof of the formula in \cite{StPi02}. We will return to this topic later in the more general framework of Tutte's polytopes, see Paragraph \ref{SoliTutti}.

\vspace{.2cm}
(c) If $\theta = -1$, the proof of Theorem~\ref{thm:main_1} amounts by \eqref{eq:cB_{n}} to the simple identities 
\begin{equation*}
   \int_{0}^1 dy_{1} \int_{0}^{1-y_{1}} dy_2 \;\cdots\; \int_{0}^{1-y_{n-1}} dy_{n}\; =\; \frac{A_{n}}{n!} 
\end{equation*}
for $n\ge 1$. This was already observed in~\cite{BeKoCa93}, where the integral on the left-hand side is viewed as the volume of the base of a pyramid. The argument in~\cite{BeKoCa93} further relies on a family of polynomials, which are recursively defined by  
\begin{equation*}
  P_{0}(x) \;= \;1\quad\text{and}\quad P_{n}(x) = \int_{0}^{1-x} P_{n-1}(t) dt
\end{equation*}
and then computed explicitly. Our argument above uses the polynomial $\cB_{n}(x) = P_{n}(1-x)$ instead and is simpler as it is based only on the straightforward observation that $\cB_{n+2}''(x) = -\cB_{n}(x)$.

\vspace{.2cm}
(d) Still in the case $\theta = -1$, Eq.\,\eqref{eq:cB_{n}} can also be stated as
\begin{align*}
\cB_{n}(1) &\; =\; {\rm Vol}\,\lacc
(z_{1}, \ldots, z_{n}) \in [0,1]^{n},\; z_{1} + z_2 \le 1,  \ldots,  z_{n-1} + z_{n} \le 1 \racc\\
&\; =\; {\rm Vol}\,\lacc
(x_{1}, \ldots, x_{n}) \in [0,1]^{n},\;  x_{1} < x_2 > x_3 < x_4 > \cdots \racc,
\end{align*}
where in the second line we have made the changes of variables $z_{i}= x_{i}$ if $i$ is odd and $z_{i}=1-x_{i}$ if $i$ is even. This implies
\begin{eqnarray*}
\cB_{n}(1) &\; =\; & \Prob\lcr X_{i}< X_{i+1} \;\text{for $i$ odd and}\; X_{i}> X_{i+1}\; \text{for $i$ even,}\; i=1,\ldots, n\rcr\\
&\; =\; & \Prob\lcr \sigma_{i}< \sigma_{i+1} \;\text{for $i$ odd and}\; \sigma_{i}> \sigma_{i+1}\; \text{for $i$ even,}\; i=1,\ldots, n\rcr,
\end{eqnarray*}
where $X_{1},\ldots,X_{n}$ are independent and uniformly distributed on $(0,1)$ and $\sigma$ denotes a permutation on $\cS_{n}$ uniformly picked at random. The fact that
\begin{equation*}
   \cB_{n}(1)\; =\; \frac{A_{n}}{n!} 
\end{equation*}
then follows from the very definition of Euler's zigzag numbers, see~\cite{Andre1881}. The stated argument is a consequence of a general result for chain polytopes given as Corollary 4.2 in~\cite{Stanley86}, the case of zigzag numbers and zigzag polytopes being discussed in Example 4.3 therein. However, it does not seem that this argument works in the case $\theta \neq -1.$
\end{Rem}

\subsection{Asymptotic behaviour}
\label{Asyb}

If $\theta = -1$, a combination of Theorem \ref{thm:main_1} with a well-known asymptotic result for Euler's zigzag numbers (see e.g.\ (1.10) in \cite{Stanley10}) yields
\begin{equation}\label{eq:Asymp0}
p_{n}^{U}(-1)\; =\; \frac{J_{n+1}(-1)}{2^{n} n!}\; =\; \frac{A_{n}}{2^{n} n!}\; \sim\; \frac{4}{\pi^{n+1}}\qquad \text{as }n\to\infty.
\end{equation}
An extension to all $\theta\in [-1,\frac{1}{2}]$, in terms of the first root
\begin{equation*}
   z_{\theta}\,=\,\inf\{z > 0:E(\theta,-z)=0\} 
\end{equation*}
of the deformed exponential function $E(\theta,z)$ defined in \eqref{eq:deformed exp} in the introduction, is provided by the next proposition.
 
\begin{Prop}\label{Asymp1}
For every $\theta\in[-1,\frac{1}{2}]$, one has
\begin{equation}\label{eq:Asymp1}
p_{n}^{U}(\theta)\; \sim \; \frac{1}{z_{\theta}\, \lambda_{\theta}^{n}}\qquad\text{as }n\to\infty,
\end{equation}
where $\lambda_{\theta} = 2(1-\theta) z_{\theta} > 1$.
\end{Prop}

\begin{proof} 
If $\theta = -1$, then $E(-1,z) = \cos z + \sin z$ and $z_{-1} = \pi/4$, so that \eqref{eq:Asymp1} matches \eqref{eq:Asymp0} above. If $\theta =0$, then $E(0,z) = 1 + z$, $z_{0} = 1$ and thus $\lambda_{0}=2$. Hence \eqref{eq:Asymp1} is again true. Finally, if $\theta \in (-1,0)\cup(0,\frac{1}{2})$, then $\vert\theta\vert < 1$, and it is easy to see that the entire function $E(\theta,\cdot)$ has order zero for these $\theta$. By the Hadamard factorization theorem, see e.g.\ Theorem XI.3.4 in \cite{Co78}, this implies
\begin{equation*}
   E(\theta,z)\; =\; \prod_{k\ge 1} \lpa 1+\frac{z}{a_{k}(\theta)}\rpa,
\end{equation*}
where $\{a_{k}(\theta), k\ge 1\}$ denotes the sequence of complex roots of $z\mapsto E(\theta,-z)$ such that 
$k\mapsto \vert a_{k}(\theta)\vert$ is positive, nondecreasing and satisfying
\begin{equation*}
   \sum_{k\ge 1} \frac{1}{\vert a_{k}(\theta)\vert^s} \; <\; \infty 
\end{equation*}
for each $s > 0$. Consequently, by \eqref{id:MR}, we obtain
\begin{align*}
\sum_{n\ge 1}(1-\theta)^{n-1} J_{n}(\theta)\, \frac{z^{n}}{n!}\ 
&=\; -\log E(\theta,-z)\ =\ -\sum_{k\ge 1} \log \lpa 1 - \frac{z}{a_{k}(\theta)}\rpa\\
&=\ \sum_{k\ge 1} \sum_{n\ge 1} \frac{z^{n}}{n(a_{k}(\theta))^{n}}\; =\; \sum_{n\ge 1} \sum_{k\ge 1} \frac{z^{n}}{n(a_{k}(\theta))^{n}}
\end{align*}
for any $z$ such that $\vert z\vert < \vert a_{1}(\theta)\vert$. Here the last equality follows by Fubini's theorem and the easily established inequality
\begin{equation*}
   \sum_{n,k\ge 1} \frac{\vert z\vert ^{n}}{n\vert a_{k}(\theta)\vert^{n}}\; \le \; \Bigg(\sum_{k\ge 1}\lva \frac{a_{1}(\theta)}{a_{k}(\theta)}\rva\Bigg)\Bigg(\sum_{n\ge 1} \frac{\vert z\vert^{n}}{\vert a_{1}(\theta)\vert^{n}}\Bigg). 
\end{equation*}
A comparison of coefficients leads to a convergent series representation of Mallows-Riordan polynomials which is valid for any $n\ge 0$ and $\theta\in(-1,1)$, namely  
\begin{equation}\label{SerRep}
\frac{J_{n+1}(\theta)}{n!}\; =\; \sum_{k\ge 1} \frac{1}{(1-\theta)^{n}a_{k}(\theta)^{n+1}}\cdot
\end{equation}
The last step is to show that $a_{1}(\theta)$ is simple and positive, and that $\vert a_2(\theta)\vert > a_{1}(\theta)$. Putting everything together, we indeed obtain
\begin{equation*}
   p_{n}^{U}(\theta)\; \sim \; \frac{1}{z_{\theta}\, \lambda_{\theta}^{n}}\qquad\text{as $n\to\infty$,} 
\end{equation*}
with $\lambda_{\theta} = 2(1-\theta) z_{\theta}>1$ since $p_{n}^{U}(\theta)\to 0$. 

In the case $\theta\in (0,1),$ the fact that $x\mapsto \theta^{x(x-1)/2}$ is a real entire function of order $2$ combined with Laguerre's criterion (see e.g.\ \cite{Polya23} p.~186 and the references therein) entails that {\em all} roots $a_k(\theta)$ are simple and positive, which is already mentioned in \cite{Sokal09}.

In the case $\theta \in (-1,0)$, the situation is less simple and we will use the following argument based on subadditivity. By the Markov property and with $\Prob_{x}=\Prob[\cdot|Y_{0}=x]$, one has 
 $$p_{n+m}^{U}(\theta)\; = \; \Erw \lcr {\bf 1}_{\{ T^U_\theta > m\}}\; \Prob_{Y_m} \lcr T^U_\theta > n\rcr\rcr$$
for all $m,n \ge 0$, and it is clear that $Y_m > 0$ a.s.\ on $\{T^U_\theta > m\}.$ On the other hand, by \eqref{eq:why_theta<=1/2} and since $\theta<0$,
\begin{align*}
\Prob_x \lcr T^U_\theta > n\rcr\ & =\ \frac{1}{2^{n}}\int_{0}^{(1+\theta x)_+}\!\int_{-\theta u_{1}}^{1}\!\!\!\cdots\int_{-(\theta u_{n-1}+\cdots+\theta^{n-1}u_{1})}^{1}du_{m}\ldots du_{2}\,du_{1}\\
&\le\ p_{n}^{U}(\theta)
\end{align*}
holds for each $x>0$, which implies subadditivity, viz.
$$ p_{n+m}^{U}(\theta)\; \leq\; p_{m}^{U}(\theta)p_{n}^{U}(\theta).$$
Let now $\{a_{j}(\theta): \, j=1,\ldots, q\}$ be the set of roots of $E(\theta,-z)$ having the same modulus as $a_{1}(\theta)$, so $\vert a_{j}(\theta) \vert = \vert a_{1}(\theta)\vert=:\rho > 0$ for each $j=1,\ldots,q$. Since $\partial_z E(\theta,-z) = -E(\theta,-\theta z)$ with $\vert \theta\vert < 1$, the minimality of $\rho$ entails $\partial_z E(\theta,-a_{j}(\theta))\neq 0 $ for each $j=1,\ldots,q$. Therefore all these first roots must be simple and we infer from \eqref{SerRep} that
\begin{equation*}
\frac{J_{n+1}(\theta)}{n!}\,=\,\frac{P(n) +o(1)}{(1-\theta)^{n} \rho^{n+1}},
\end{equation*}
with 
$$ P(n) \, =\, \sum_{j=1}^{q}e^{2\pi {\rm i} \alpha_j (n+1)} $$
for some $q\ge 0$ and distinct frequencies $\alpha_j \in [0,1)$ for $j=1,\ldots, q.$ Set $\cA =\{\alpha_j: \, j= 1,\ldots , q\}$ and suppose $0\not\in\cA.$ By Lemma 4 in \cite{Braverman06}, there exists a constant $c < 0$ independent of $n$ such that $P(n) \le c $ for infinitely many $n$. Using Theorem \ref{thm:main_1}, this implies that
$$ 2^{n}(1-\theta)^n \rho^{n+1} p_n^U(\theta)\ =\ P(n) + o(1)\ \le\ c/2\ <\ 0$$
infinitely often which is impossible. Hence, we have $0\in \cA$ and can choose $\alpha_1 = 0.$ Assuming $\cA \neq\{0\}$, thus $q\ge 2$, we next observe that
$$(2(1-\theta))^{n+m} \rho^{n+m+1} ( p_n^U(\theta) p_m^U(\theta) - p_{n+m}^U(\theta)) = P(n)P(m) - P(n+m) + o(1)$$
with a decomposition of the real quantity $P(n)P(m) - P(n+m)$ as a sum of two real sums
$$\sum_{j=2}^{q}e^{2\pi {\rm i} \alpha_j (m+1)}\; +\; \sum_{j=2}^{q} \lpa\sum_{k\neq j} e^{2\pi {\rm i} \alpha_k (m+1)}\!\rpa\! e^{2\pi {\rm i} \alpha_j (n+1)}.$$
Since $\alpha_j \in (0,1)$ for $j=2,\ldots, q$, we invoke again Lemma 4 in \cite{Braverman06} to infer the existence of $m_0\ge 0$ and $c < 0$ such that 
$$ \sum_{j=2}^{q}e^{2\pi {\rm i} \alpha_j (m_0+1)}\, \le\, c. $$
From this,we deduce
$$ P(n)P(m_0) - P(n+m_0)\ \le\ c + \sum_{j=2}^{q} \lpa\sum_{k\neq j} e^{2\pi {\rm i} \alpha_k (m_0+1)}\!\rpa\! e^{2\pi {\rm i} \alpha_j (n+1)}\ \le\ c $$
for infinitely many $n,$ again by Lemma 4 in \cite{Braverman06}. This leads to
$$(2(1-\theta))^{n+m_0} \rho^{n+m_0+1} ( p_n^U(\theta) p_{m_0}^U(\theta) - p_{n+m_0}^U(\theta))\ \le\ c/2\ <\ 0$$
for infinitely many $n,$ contradicting the above subadditivity property. Hence $\cA =\{ 0\}$ must hold, which means that $a_{1}(\theta)$ is simple and positive and that $\vert a_2(\theta)\vert > a_{1}(\theta)$ as required.
\end{proof}

\begin{Rem}\label{Conv}\rm
(a) The following formula is a consequence of \eqref{SerRep} and provides a full asymptotic expansion of $p_{n}^{U}(\theta)$ as $n\to\infty$:
\begin{equation*}
   p_{n}^{U}(\theta)\; =\; \frac{1}{z_{\theta}\, \lambda_{\theta}^{n}}\; +\; \sum_{k\ge 2} \frac{1}{2^{n}(1-\theta)^{n}a_{k}(\theta)^{n+1}}. 
\end{equation*}
It remains valid at $\theta = 0$ with $a_{k}(\theta):=\infty$ for all $i\ge 2$, and also at $\theta =-1$ with $a_{k}(\theta):= (-1)^{k-1}(2k-1)\pi/4$ for all $i\ge 2$, then boiling down to the well-known formula
\begin{equation*}
   \frac{A_{n}}{n!}\; =\; 2\lpa\frac{2}{\pi}\rpa^{n+1} \sum_{k\ge 0} \frac{(-1)^{k(n+1)}}{(2k+1)^{n+1}}
\end{equation*}
for Euler's zigzag numbers, see e.g.\ the end of Section 1 in \cite{Stanley10}. For odd $n$, the latter amounts to the classical formula for $\zeta(n+1)$ in terms of Bernoulli numbers. The roots $a_{k}(\theta)$ are generally non-explicit, but in the case $\theta \in(0,1)$, the above proof and Theorem \ref{thm:main_1} imply the asymptotic result
\begin{equation}\label{AsympJn}
\frac{J_{n+1}(\theta)}{n!}\;\sim\; \frac{1}{z_{\theta}}\lpa\frac{2}{\lambda_{\theta}}\rpa^{n}\qquad\text{as $n\to\infty$,}
\end{equation}
which is valid for every $\theta\in[-1,1)$ and has apparently remained unnoticed in the literature. Finally, we mention that $\lambda_{0} = 2$, and we will show in Remark~\ref{Asympb} that $\lambda_{\theta} > 2$ for $\theta\in[-1,0)$ and $\lambda_{\theta} < 2$ for $\theta\in (0,1)$.

\vspace{.2cm}
(b) As a consequence of Remark \ref{Inc1} and \eqref{SerRep}, the function $\theta\mapsto (1-\theta)z_{\theta}$ is nonincreasing on $[-1,\frac{1}{2}]$ and taking values $\frac{\pi}{2},\,1$ and $\frac{1}{2}\lambda_{\frac{1}{2}}$ at $-1,\,0$ and $\frac{1}{2}$, respectively. Since $J_{n}(\theta)$ has  positive coefficients, it remains nonincreasing on $[\frac{1}{2},1)$, with limit $\frac{1}{e}$ at 1 by Stirling's formula and the fact that $J_{n}(1) = n^{n-2}$. Stirling's for\-mula may also be used to provide a polynomial correction in  \eqref{AsympJn} for $\theta = 1$. It would be interesting to know if this monotonicity property of the first negative zero of the deformed exponential function can be obtained directly without connection to Mallows-Riordan polynomials and persistence probabilities.

\vspace{.2cm}
(c) For $\theta\in (0,1)$, complete asymptotic expansions of the roots of the deformed exponential function were recently obtained in \cite{WZ18}, showing in particular $a_{k}(\theta) \sim k\theta^{-(k-1)}$ as $k\to\infty$. For $\theta \in(-1, 0)$, we believe that the roots are simple and alternate in sign as for $\theta = -1$, but we have no proof. See also \cite{Sokal09} for several conjectures on the zeroes of the deformed exponential function with a parameter $\theta$ in the unit disk.

\vspace{.2cm}
(d) For $\theta > 0,$ it is easily shown with the above argument that the sequence $\{p^U_n(\theta)\}_{n\ge 0}$ is superadditive, in other words that
$$p^U_{n +m}(\theta)\ge p^U_n(\theta)p^U_m(\theta)$$
for all $m,n \ge 0.$ In the case $\theta \in (0,1/2],$ this is also the consequence of the stronger property that the sequence is log-convex, see Proposition \ref{ID2} below.  
\end{Rem}

\section{The case $\theta <0$}

\subsection{Proof of Theorem \ref{thm:main_2}} 
\label{sec:case_<-1}

Here we embark on the general integral formula
\begin{equation}
\label{MecGen}
p_{n}(\theta)\; =\; \int_{0}^{\infty}\!\!\! \int_{-\theta u_{1}}^{\infty} \cdots\;\int_{-(\theta u_{n-1} + \cdots + \theta^{n-1} u_{1})}^{\infty}\!_!\!\! dF(u_{1})\,\ldots\, dF(u_{n}),
\end{equation}
valid for all $\theta\in\R$, where $F$ denotes the common distribution function of the i.i.d. innovations $X_{1},X_{2},\ldots$ In the case $\theta <0,$ the domain of integration is a subset of the positive orthant $\{u_1 \ge 0, \ldots, u_n \ge 0\}$. As a consequence, the persistence probabilities $p_n(\theta)$ and $p_n(1/\theta)$ only depend on $F$ restricted to the nonnegative halfline for all $n\ge 1.$ We can also assume $F(0-) = F(0) < 1$ because otherwise all probabilities are zero. Setting $c = 1 - F(0) > 0$ and 
$$ G(x) \, = \, 1 - G(-x) \, = \, \frac{1}{2}\, +\, \frac{F(x) - F(0)}{2c} $$ 
for all $x\ge 0$, we now see that 
$$p_{n}(\theta)\; =\; (2c)^n\, \int_{0}^{\infty}\!\!\! \int_{-\theta u_{1}}^{\infty} \cdots\;\int_{-(\theta u_{n-1} + \cdots + \theta^{n-1} u_{1})}^{\infty}\!_!\!\! dG(u_{1})\,\ldots\, dG(u_{n})$$
for every $n\ge 0$ and $\theta < 0$, and since $G$ is symmetric and continuous, we can assume without loss of generality that $F$ itself has these properties.
Setting again $r = -1/\theta > 0$, we have
\begin{align*}
p_{n}(\theta)\ &=\ \Prob\left[\sum_{j=1}^{k} \theta^{k-j} X_{j}\ge 0, \; k=1,\ldots, n \right]\\
&=\ \Prob\left[\sum_{j=1}^{k}(-1)^{k-j}r^{j-1} X_{j}\ge 0, \; k=1,\ldots, n \right]\\
&=\ \Prob\left[(-1)^{k-1}\sum_{j=1}^{k} r^{j-1} X_{j}\ge 0, \; k=1,\ldots, n \right]
\end{align*}
where $X_{j}\eqdist -X_{j}$ for even $j$ and the mutual independence of the $X_{j}$ has been used in the last line. Defining $S_{k}=X_{1}+\cdots+ r^{k-1}X_{k}$, $A_{k}^{+}=\{S_{k}\ge 0\}$ and $A_{k}^{-} = \{S_{k}\le 0\}$ for $k=1,\ldots, n$, we see that, with writing $AB$ as shorthand for $A\cap B$,
\begin{align*}
p_{n}(\theta)\ =\ 
\begin{cases}
\Prob\left[A_{1}^{+}A_{2}^{-}\cdots A_{n}^{-}\right]&\text{if $n$ even,}\\
\Prob\left[A_{1}^{+}A_{2}^{-}\cdots A_{n}^{+}\right]&\text{if $n$ odd}.
\end{cases}
\end{align*}
On the other hand, we deduce from \eqref{MecGen} upon the successive change of variables $u_{k}= (-1)^{k-1}v_{k}$ for $k=1,\ldots, n$ that
\begin{align*}
p_{k}(1/\theta)\ =\ 
\begin{cases}
{\displaystyle \int_{0}^{\infty}\!\!\int_{-\infty}^{-r v_{1}}\!\!\!\!\cdots\int_{-\infty}^{-(r^{k-1}v_{1}+\cdots+rv_{k-1})}\hspace{-.7cm}dG(v_{1})\ldots dG(v_{k})} &\text{if $k$ is even,}\\[4mm]
\hfill {\displaystyle \int_{0}^{\infty}\!\!\int_{-\infty}^{-r v_{1}}\!\!\!\!\cdots\int_{-(r^{k-1}v_{1}+\cdots+rv_{k-1})}^{\infty}\hspace{-.7cm} dG(v_{1})\ldots dG(v_{k})} &\text{if $k$ is odd}.
\end{cases}
\end{align*}
Suppose first that $n$ is odd. Using the continuity of $G$, we obtain
\begin{align*}
p_{n}(\theta)\ &=\ \Prob\left[A_{1}^{+}A_{2}^{-}\cdots A_{n}^{+}\right] \\
&=\ \Prob\left[A_{1}^{+}A_{2}^{-}\cdots A_{n-2}^{+}\cap\{-r^{n-1} X_{n}\le S_{n-1} \le 0\}\right]\\
&=\ \int_{0}^{\infty} \Prob\left[A_{1}^{+}A_{2}^{-}\cdots A_{n-2}^{+}\cap\{S_{n-1}\in [-r^{n-1}v_{1},0]\}\right]\,dG(v_{1})\\
&=\ \frac{p_{n-1}(\theta)}{2}\,-\int_{0}^{\infty}\Prob\left[A_{1}^{+}A_{2}^{-}\cdots A_{n-2}^{+}\cap\{S_{n-1} \le -r^{n-1}v_{1}\}\right]\,dG(v_{1})\\
&=\ p_{n-1}(\theta) p_1(1/\theta)\,-\int_{0}^{\infty}\Prob\left[A_{1}^{+}A_{2}^{-}\cdots A_{n-2}^{+}\cap\{S_{n-1} \le -r^{n-1}v_{1}\}\right]\,dG(v_{1}).
\end{align*}
As for the integrand in the last line, we further compute
\begin{align*}
&\Prob\left[A_{1}^{+}A_{2}^{-}\cdots A_{n-2}^{+}\cap\{S_{n-1} \le -r^{n-1} v_{1}\}\right]\\
&\; =\; \Prob\left[A_{1}^{+}A_{2}^{-}\cdots A_{n-3}^{-}\cap\{0 \le S_{n-2} \le -r^{n-1} v_{1}-r^{n-2}X_{n-1}\}\right]\\
&\; =\;\int_{-\infty}^{-rv_{1}}\Prob\left[A_{1}^{+}A_{2}^{-}\cdots A_{n-3}^{-}\cap\{0\le S_{n-2} \le -s_2\}\right]\, dG(v_2)\\
&\; =\; p_{n-2}(\theta)F(-r v_{1})\,-\,\int_{-\infty}^{-rv_{1}}\hspace{-5pt}\Prob\left[A_{1}^{+}A_{2}^{-}\cdots A_{n-3}^{-}\cap\{S_{n-2}\ge-s_2\}\right]\,dG(v_2),
\end{align*}
with the notation $s_{k} = r^{n-1} v_{1}+\ldots+ r^{n-k}v_{k}$ for $k=1,\ldots,n$. Therefore, integrating with respect to $v_{1}$ and using the previous formula for $p_{2}(1/\theta),$ we obtain
\begin{align*}
p_{n}(\theta)\ &=\ p_{n-1}(\theta)p_{1}(1/\theta)\,-\,p_{n-2}(\theta)p_2(1/\theta)\\
&\quad +\;\int_{0}^{\infty}\!\!\int_{-\infty}^{-rv_{1}}\Prob\left[A_{1}^{+}A_{2}^{-}\cdots A_{n-3}^{-}\cap\{S_{n-2}\ge-s_2\}\right]dG(v_{1})\,dG(v_2).
\end{align*}
Computing in the same manner the integrand in the previous line, we get 
\begin{align*}
&\Prob\left[A_{1}^{+}A_{2}^{-}\cdots A_{n-3}^{-}\cap\{S_{n-2}\ge-s_2\}\right]\\
&\qquad\quad =\;\Prob\left[A_{1}^{+}A_{2}^{-}\cdots A_{n-4}^{+}\cap\{0\ge S_{n-3}\ge-s_2-r^{n-3}X_{n-2}\}\right]\\
&\qquad\quad =\;\int_{-r^{2}v_{1} -rv_2}^{\infty}\!\!\!\!\Prob\left[A_{1}^{+}A_{2}^{-}\cdots A_{n-4}^{+}\cap\{0\ge S_{n-3}\ge-s_3\}\right]\, dG(v_3) \\
& \qquad\quad = \; p_{n-3}(\theta) F(r^{2} v_{1}+rv_2)\\
& \qquad\qquad\quad -\,\int_{-r^{2}v_{1} -rv_2}^{\infty}\!\!\!\Prob\left[A_{1}^{+}A_{2}^{-}\cdots A_{n-4}^{+}\cap\{S_{n-3} \le -s_3\}\right]\,dG(v_3),
\end{align*}
where the symmetry of $F$ has been utilized for the final equality. Using the previous formula for $p_{3}(1/\theta),$ we thus arrive at
\begin{gather*}
p_{n}(\theta)\ =\ p_{n-1}(\theta)p_{1}(1/\theta)\,-\,p_{n-2}(\theta)p_2(1/\theta)\, +\,p_{n-3}(\theta)p_3(1/\theta)\\
-\int_{0}^{\infty}\!\!\!\int_{-\infty}^{-rv_{1}}\!\!\!\int_{-r^{2}v_{1} -rv_2}^{\infty}\!\!\!\!\!\!\!\Prob\left[A_{1}^{+}A_{2}^{-}\cdots A_{n-4}^{+}\cap\{S_{n-3} \le -s_3\}\right]\ \prod_{i=1}^{3}dG(v_{i}).
\end{gather*}
Continuing this way, or by an induction, we arrive at the identity
\begin{equation*}
   p_{n}(\theta)\; =\; \sum_{k=1}^{n} (-1)^{k-1} p_{n-k}(\theta)p_{k}(1/\theta), 
\end{equation*}
which is the desired result for odd $n$ because $p_{0}(1/\theta) = 1$. The argument for even $n$ follows analogously and is therefore omitted.

\qed

\begin{Rem}
\label{Atomic}
{\em When the innovation law has atoms on $[0,\infty),$ the statement of Theorem \ref{thm:main_2} is not true in general: if $dF(x) = c \mu(dx) +(1-c) \delta_0(dx)$ with $c\in (0,1)$ and $\mu$ some probability on $(-\infty, 0),$ then the right-hand side of \eqref{eq:duality theta<0} equals $(1-c)^n (1+ (-1)^n)/2.$} 
\end{Rem}

\subsection{Behavior at $\theta=-1$}

It is clear from the polynomial identities \eqref{eq:MR-id} and \eqref{eq:main_theorem_2} stated in Theorem~\ref{thm:main_1} and Corollary \ref{cor:Thm 2}, respectively, that the mapping $\theta\mapsto p_{n}^{U}(\theta)$ is smooth on $(-\infty,-1)\cup (-1,\frac{1}{2})$ for any $n\ge 0$. Regarding the natural question of  its behavior at $\theta = -1$, we prove the following result.

\begin{Prop}
\label{Phase1}
For each $n\ge 1$, the mapping $\theta\mapsto p_{n}^{U}(\theta)$ is $\cC^1\!$ at $\theta = -1$.  
\end{Prop}

\begin{proof}
Continuity follows directly from $2^{n} n!\,p_{n}^{U}(\theta) \to \wJ_{n+1}(-1)$ as $\theta\uparrow -1$ and $\wJ_{n+1}(-1) = J_{n+1}(-1)$, which in turn is a consequence of
\begin{align*}
\sum_{n\ge 0} \wJ_{n+1}(-1)\,\frac{z^{n}}{n!}\ &=\ \lpa\sum_{n\ge 0} (-1)^{n} J_{n+1}(-1)\,\frac{z^{n}}{n!}\rpa^{\!\! -1}\\
&=\ \frac{\cos z}{1-\sin z}\; = \; \frac{1+\sin z}{\cos z}\; =\;\sum_{n\ge 0} J_{n+1}(-1)\,\frac{z^{n}}{n!}.
\end{align*}
To prove continuity of the derivative, we introduce the generating functions
\begin{equation*}
   J(\theta, z) \; =\; \sum_{n\ge 0} J_{n+1}(\theta)\,\frac{z^{n}}{n!}\qquad \text{and}\qquad \wJ(\theta, z) \; =\; \sum_{n\ge 0} \wJ_{n+1}(\theta)\,\frac{z^{n}}{n!}\cdot
\end{equation*}
Since $\wJ(\theta, z) J(\theta, -z) = 1$ by \eqref{eq:modified_MR} and therefore
\begin{equation*}
   J(\theta, -z)\,\ptha \wJ(\theta, z)\,+\,\wJ(\theta, z)\, \ptha J(\theta, -z)\ =\ 0, 
\end{equation*}
we infer
\begin{equation}
\label{Pth}
\ptha \wJ(-1, z)\; =\; -(\wJ(-1, z))^{2}\, \ptha J(-1,-z)\; =\; -\frac{\cos^{2}\! z\; \ptha J(-1,-z)}{(1-\sin z)^{2}}\cdot
\end{equation}
Next, we want to find $\ptha J(-1,z)$ with the help of the formula
\begin{gather*} 
J(\theta, z) \; =\; \frac{F(\theta, z)}{G(\theta,z)},
\shortintertext{where}
F(\theta, z)\; =\; \sum_{n\ge 0} \frac{\theta^{n(n+1)/2}}{(\theta -1)^{n}}\,\frac{z^{n}}{n!}\quad\text{and}\quad G(\theta, z)\; =\; \sum_{n\ge 0} \frac{\theta^{n(n-1)/2}}{(\theta -1)^{n}}\,\frac{z^{n}}{n!}.
\end{gather*}
Direct computation provides $F(-1,z) = c+s,\, G(-1,z) = c-s$,
\begin{gather}
\partial_{\theta} F(-1,z)\; =\; \frac{z(s-c)}{4}\, + \,\frac{z^{2}(c+s)}{8}\nonumber
\shortintertext{and}
\partial_{\theta} G(-1,z)\; =\; \frac{-z(s+c)}{4}\, +\, \frac{z^{2}(c-s)}{8}\label{partG}
\end{gather}
with $c:=\cos(z/2)$ and $s:=\sin(z/2)$. After some trigonometric simplifications, this yields
\begin{equation*}
   \partial_{\theta}J(-1, z) \; = \;\frac{ G(-1,z)\,\ptha F(-1,z)\, - \,F(-1,z)\, \ptha  G(-1,z) }{G^{2}(-1,z)}\; = \; \frac{z\, \sin z}{2(1-\sin z)}
\end{equation*}
and, by using this in \eqref{Pth}, we finally obtain
\begin{equation*}
   \ptha \wJ(-1, z)\; =\; -\frac{z\, \sin z}{2(1-\sin z)}\; =\; - \,\ptha J(-1,z).
\end{equation*} 
Therefore $\wJ_{n+1}'(-1) = - J_{n+1}(-1)$ for all $n\ge 0$ by comparing coefficients. 
Since $(p_{n}^{U})'(\theta)\to - \wJ_{n+1}'(-1)$ as $\theta\uparrow -1$ and $(p_{n}^{U})'(\theta)\to J_{n+1}'(-1)$ as $\theta\downarrow -1$, the proof is complete.\qed
\end{proof}

\begin{Rem}\rm
The formula
\begin{equation*}
   \sum_{n\ge 0} J_{n+1}'(-1)\,\frac{z^{n}}{n!}\; =\; \partial_{\theta} J(-1, z) \; = \;\frac{z\,\sin z}{2(1-\sin z)}\; =\; \frac{z\,\tan z}{2}\, J(-1,z)
\end{equation*}
reveals, after some additional computations, the curious fact that $J_{n+1}'(-1) = 0$ for $n\in\{0,1\}$ and $J_{n+1}'(-1) = (n/2)\, J_{n+1}(-1)$ for all $n\ge 2$, for which we could not find a reference.
\end{Rem}
 
Our next proposition exhibits the remarkable property that, for any $n\ge 2$, the mapping $\theta \mapsto p_{n}^{U}(\theta)$ is not $\cC^{2}\!$ at $\theta = -1$, which means that a phase transition of order $2$ occurs for these persistence probabilities at the common boundary of the two regimes $\theta \in(-\infty, -1)$ and $\theta \in (-1,\frac{1}{2}).$

\begin{Prop}
\label{Phase2}
For any $n\ge 2$, the mapping $\theta\mapsto (p_{n}^{U})''(\theta)$ is not continuous at $\theta = -1$.  
\end{Prop}

\begin{proof}
We know that $2^{n} n!\, (p_{n}^{U})''(\theta)\to J_{n+1}''(-1)$ as $\theta\downarrow -1$ and, by recalling $2^{n} n!p^{\,U}_{n}(\theta) = \wJ_{n+1}(1/\theta)$ for $\theta < -1$, we also have 
\begin{equation*}
   2^{n} n!\, (p_{n}^{U})''(\theta)\to \wJ_{n+1}''(-1) + 2 J_{n+1}'(-1)\qquad\text{as $\theta\uparrow -1.$}
\end{equation*}
We will now prove that $\wJ_{n+1}''(-1) + 2 J_{n+1}'(-1) > J_{n+1}''(-1)$ for all $n\ge 2$ by first computing the corresponding exponential generating function 
\begin{equation*}
   \partial^{2}_{\theta\theta} \wJ(-1,z)\,-\, \partial^{2}_{\theta\theta} J(-1,z)\,+\, 2\, \ptha J(-1,z)
\end{equation*}
and then showing that its coefficients of order greater than $2$ are positive. First, by differentiating twice the relation
\begin{equation*}
   \frac{J(\theta, z)}{\wJ(\theta,z)}\; =\; J(\theta,z) J(\theta,-z)
\end{equation*}
with respect to $\theta$ and, recalling 
$$J(-1,z) = \wJ(-1,z)\qquad \mbox{and} \qquad \ptha J(-1,z) = -\ptha \wJ(-1,z),$$
we obtain 
\begin{align*}
\partial^{2}_{\theta\theta} \wJ(-1,z)\,&-\,\partial^{2}_{\theta\theta} J(-1,z)\ =\ \frac{4(\ptha J(-1,z))^{2}}{\wJ(-1,z)}\\ 
&-\ J(-1,z)\lpa J(-1,z) \partial^{2}_{\theta\theta}J(-1,-z) + J(-1,-z)\partial^{2}_{\theta\theta} J(-1,z)\rpa\\
&-\,2 J(-1,z)\,\ptha J(-1,z)\,\ptha J(-1,-z).
\end{align*}
It is easy to deduce from the proof of Proposition \ref{Phase1} that
\begin{equation*}
   \frac{4(\ptha J(-1,z))^{2}}{\wJ(-1,z)}\,-\, 2 J(-1,z)\,\ptha J(-1,z)\,\ptha J(-1,-z)\; =\; \frac{z^{2} \sin z\, \tan z}{2(1-\sin z)}. 
\end{equation*}
On the other hand, we have
\begin{gather*}
J(-1,z) \partial^{2}_{\theta\theta}J(-1,-z)\,+\,J(-1,-z)\partial^{2}_{\theta\theta} J(-1,z)\; = \; A(z) \, +\, A(-z),
\shortintertext{where} 
A(z) \; =\; \frac{\partial^{2}_{\theta\theta}J(-1,z)}{J(-1,z)}.
\end{gather*}
Using the notation of Proposition \ref{Phase1}, this function can be decomposed as
\begin{equation*}
   \frac{G(-1,z)\partial^{2}_{\theta\theta}F(-1,z) - F(-1,z)\partial^{2}_{\theta\theta}G(-1,z)}{F(-1,z)G(-1,z)}\, -\, \frac{2\ptha G(-1,z)\ptha J(-1,z)}{F(-1,z)} 
\end{equation*}
and we find with the help of \eqref{partG} that
\begin{equation*}
   \frac{\ptha G(-1,z)\ptha J(-1,z)}{F(-1,z)}\,+\,\frac{\ptha G(-1,-z)\ptha J(-1,-z)}{F(-1,-z)}\; =\; \frac{z^{3}\tan z}{8}\, - \,\frac{z^{2} \tan^{2}\! z}{4}. 
\end{equation*}
A combination of the previous facts leads to the simplified formula 
\begin{equation*}
   \partial^{2}_{\theta\theta} \wJ(-1,z)\, -\, \partial^{2}_{\theta\theta} J(-1,z) \; = \;\frac{z^3 \sin z}{4(1-\sin z)}\, +\, \frac{B(z) + B(-z)}{1- \sin z}, 
\end{equation*}
where $B(z)=F(-1,z)\,(\partial^{2}_{\theta\theta}G(-1,z)-\partial^{2}_{\theta\theta}F(-1,-z))$. After some further elementary but tedious calculations, one finally arrives that
\begin{equation*}
   \partial^{2}_{\theta\theta}G(-1,z)\, - \, \partial^{2}_{\theta\theta}F(-1,-z)\; =\;\frac{z^{2}(c-s)}{2}\, -\, \frac{z(1+z^{2}/4)(c+s)}{2}, 
\end{equation*}
with $c$ and $s$ as in \eqref{partG}. By finally putting everything together, we obtain the simple expression
\begin{equation*}
   \partial^{2}_{\theta\theta} \wJ(-1,z)\, -\, \partial^{2}_{\theta\theta} J(-1,z)\, +\, 2\ptha J(-1,z) \; = \;\frac{z^{2}(1+\sin z)}{\cos z} 
\end{equation*}
for the generating function of $(p_{n}^{U})''(-1_-) - (p_{n}^{U})''(-1_+)$ whose coefficients of order $\ge 2$ are positive as required. More precisely, we have
\begin{equation*}
   (p_{n}^{U})''(-1_-)\; -\; (p_{n}^{U})''(-1_+)\; =\; \frac{p^{\,U}_{n-2}(-1)}{4} 
\end{equation*}
for each $n\ge 2$. Of course, $(p_{n}^{U})''(-1_\pm)$ denote right-hand/left-hand limits of $(p_{n}^{U})''$ at $-1$.
\qed
\end{proof}

\begin{Rem}\label{Expo}\rm
The nonsmoothness of $p_{n}^{U}(\theta)$ at $\theta = -1$ is related to the choice of the uniform distribution as the innovation law, more precisely to the nonsmoothness of its density at the boundaries of its support. Choosing instead the biexponential law with density $e^{-\vert x\vert}/2$ for the innovations, it can be easily derived from (\ref{MecGen}) that
\begin{equation*}
   \tp_{n}(\theta)\; =\; \frac{1}{(1-\theta)^{n-1}} 
\end{equation*}
for any $n\ge 1$ and $\theta\le 0$, which is obviously smooth. Let us further note in passing that one can easily check the assertion of Theorem~\ref{thm:main_2} for the persistence probabilities in this case. We will return to biexponential innovations in the more complicated case $\theta > 0$ in Remark \ref{MRGL}(a).
\end{Rem}

\subsection{Proof of Corollary \ref{cor:Thm 2} and properties of the $\wJ_{n}(\theta)$}
\label{SpwJ}

The proof of Corollary \ref{cor:Thm 2} essentially amounts to a derivation of the stated properties of the $\wJ_{n}(\theta)$ defined by \eqref{eq:modified_MR}, which is done by Proposition \ref{Misc} below. We also collect a number of further relevant aspects of the $\wJ_{n}(\theta)$ in this subsection, though without claiming to be exhaustive. Eq.\,\eqref{eq:main_theorem_2}, where these polynomials appear, can easily be deduced from Theorems \ref{thm:main_1} and \ref{thm:main_2} via an induction, and we omit giving details. 
 
\begin{Prop}\label{Misc}
For each $n\ge 1$, $\wJ_{n+1}(\theta) $ is a polynomial in $\Z[X]$ of degree $n(n-1)/2$ and with valuation $n-1$. Moreover, $(-\theta)^{-(n-1)}\wJ_{n+1}(\theta)\in\Z[X]$ has positive coefficients.
\end{Prop}

\begin{proof}
By definition \eqref{eq:modified_MR} and Cauchy's product, the $\wJ_{n}(\theta)$ are given in terms of the Mallows-Riordan polynomials by the recursive formula
\begin{equation}\label{Jrec}
\wJ_{n+1} (\theta)\; =\; \sum_{k=1}^{n} (-1)^{k-1} \binom{n}{k} J_{k+1}(\theta)\wJ_{n+1-k} (\theta)
\end{equation}
for all $n\ge 1$, with initial condition $\wJ_{1}(\theta) = 1$. Via induction, \eqref{Jrec} readily implies that $\wJ_{n+1}(\theta) $ is indeed a polynomial in $\Z[X]$ of degree $n(n-1)/2$ for any $n\ge 0$. Regarding the exponential generating function of the family $\{\wJ_{n+1}(\theta),\, n\ge 0\}$, it follows from \eqref{eq:modified_MR} and \eqref{id:Mathar} that
\begin{equation}
\label{expwJ}
\sum_{n \ge 0} \wJ_{n+1}(\theta)\, \frac{z^{n}}{n!}\; =\; \exp\lcr \sum_{n\ge 1}
(-1)^{n-1}  (1 + \theta + \cdots + \theta^{n-1})\, J_{n}(\theta)\,\frac{z^{n}}{n!}\rcr\!.\end{equation}
Differentiating with respect to $z$ and applying again Cauchy's product, the following alternative recursion similar to \eqref{recMR} is obtained: 
\begin{equation}
\label{recwJ}
\wJ_{n+2}(\theta)\; =\; \sum_{k=0}^{n} \binom{n}{k} (-1)^{k} (1 + \theta + \cdots + \theta^{k})\, J_{k+1}(\theta)\, \wJ_{n+1-k} (\theta)
\end{equation}
for any $n\ge 0$. Let us define $A_{k}(\theta) =(1 + \theta + \cdots + \theta^k)J_{k+1}(\theta)/k!$ for $k\ge 0$ and $B_{k}(\theta) = (-1)^{k-1}\wJ_{k+1}(\theta)/k!$ for $k\ge 1$. Then we infer
\begin{equation*}
   B_{n+1}(\theta) \; =\; \frac{1}{n+1} \lpa A_{n} (\theta)\, -\, \sum_{k=0}^{n-1} A_{k}(\theta)\, B_{n-k}(\theta)\rpa
\end{equation*}
which, by making use of the above recursion \eqref{Jrec}, easily leads to
\begin{equation}
\label{Brec}
B_{n+1}(\theta) \; =\; \frac{1}{n+1} \, \sum_{k=0}^{n-1} \frac{(\theta^{k+1} +\cdots +\theta^{n})}{k!}\, J_{k+1}(\theta)\, B_{n-k}(\theta)
\end{equation}
for any $n\ge 1$. Finally recalling that all coefficients of the $J_{n+1}(\theta)$ are positive, all the other asserted properties of the $\wJ_{n+1}(\theta)$ follow by an induction.\qed
\end{proof}

\begin{Rem}\rm
As $A_{k}(0) = 1$ and $\wJ_2(\theta) = 1$, it follows from (\ref{Brec}) by an induction that
\begin{equation}\label{eq:limit tp_{n}^U}
\tp_{n}^{\,U}(-\theta)\; =\; \frac{\wJ_{n+1}(-1/\theta)}{n!}\;\sim\; \frac{1}{2\theta^{n-1}}\qquad\text{as }\theta\to \infty
\end{equation}
for any $n\ge 2$, which could also be inferred from \eqref{MecGen}, but only after tedious calculations. Considering
\begin{equation*}
   C_{n}(\theta)\;= \;(-\theta)^{-(n-1)}\wJ_{n+1}(\theta) 
\end{equation*}
for $n\ge 1$, which is a polynomial of degree $(n-1)(n-2)/2$ with only positive coefficients, \eqref{eq:limit tp_{n}^U} provides that $C_{n}(0) = n!/2$ for all $n\ge 2$, whereas $C_{1}(0) = 1$. Moreover, it has leading coefficient 1 as one can deduce from \eqref{Jrec} by a straightforward argument. For their coefficients of higher degree, the polynomials $C_{n}$ exhibit some similarities with Touchard's polynomials, see~\cite[p.\,24]{Touchard52},  but their full combinatorics have apparently not yet been discussed in the literature. 
\end{Rem}

An alternative expression for the exponential generating function of the family $\{\wJ_{n+1}(\theta),\, n\ge 0\}$ is 
\begin{equation}
\label{AltExp}
\sum_{n\ge 0} \wJ_{n+1}(\theta)\, \frac{z^{n}}{n!}\; \fps \;   \frac{{\displaystyle \sum_{n \ge 0} \frac{\theta^{n(n-1)/2}}{(1-\theta)^{n}}\, \frac{z^{n}}{n!}}}{{\displaystyle \sum_{n \ge 0} \frac{\theta^{n(n+1)/2}}{(1-\theta)^{n}}\, \frac{z^{n}}{n!}}}
\end{equation}
but it does not even seem to provide directly that $\wJ_{n+1}(\theta)$ is a polynomial. Choosing $\theta = -1$, we obtain 
\begin{equation*}
   \sum_{n\ge 0} \wJ_{n+1}(-1)\, \frac{z^{n}}{n!}\; =\; \frac{\cos z}{1-\sin z}\; =\; \frac{1 +\sin z}{\cos z}\; =\; \sum_{n\ge 0} J_{n+1}(-1)\, \frac{z^{n}}{n!}, 
\end{equation*}
which has been observed already in the proof of Proposition \ref{Phase1} and implies
\begin{equation*}
   \wJ_{n+1}(-1)\; =\; J_{n+1}(-1) \; =\; A_{n} 
\end{equation*}
for all $n\ge 0$. The following result shows another striking similarity of the $\wJ_{n}$ with Mallows-Riordan polynomials. Recall that $J_{n+1}(1) = (n+1)^{n-1}$ for any $n\ge 0$, see e.g.~\cite{Kreweras80} after Eq.\,(3).

\begin{Prop}\label{Strike}
For any $n\ge 0$, $(-1)^{n-1}\wJ_{n+1}(1) = (n-1)^{n-1}$ holds. 
\end{Prop}

\begin{proof}
It follows from \eqref{expwJ} for $\theta = 1$ that 
\begin{align*}
\sum_{n \ge 0}\wJ_{n+1}(1)\, \frac{z^{n}}{n!} \; &=\; \exp\lcr \sum_{n\ge 1}
(-1)^{n-1} \, n \,J_{n}(1)\,\frac{z^{n}}{n!}\rcr\\
&=\; \exp\lcr \sum_{n\ge 1} (-n)^{n-1} \,\frac{z^{n}}{n!}\rcr\; =\; \exp \,[W(z)]\\
\end{align*}
for any $z\in (-1/\e,1/\e)$. Here $W(z)$ equals the Lambert function $W(z)$ and is known to satisfy 
\begin{equation*}
   e^{-W(z)}\, = \,z^{-1}W(z) 
\end{equation*} 
on $(-1/\e, 1/\e)$. A comparison of coefficients on both sides of this equation yields the desired result.\qed
\end{proof}

We conclude this paragraph with a monotonicity property on $[-1,0]$ that is similar to the one stated in Remark \ref{Inc1} for Mallows-Riordan polynomials. It is clear from Proposition \ref{Misc} that $\theta\mapsto (-1)^{n}\wJ_{n}(\theta)$ increases on $\R^+$ for any $n\ge 3$, and we have also observed in the proof of Proposition \ref{Phase1} that 
\begin{equation*}
   \wJ_{n+1}'(-1)\, =\, -J'_{n+1}(-1) \,=\, -nA_{n}/2 \,<\, 0
\end{equation*} 
for any $n\ge 2$. The following property, obtained through the connection with persistence probabilities, completes the picture.

\begin{Prop}
\label{Inc2}
The function $\theta\mapsto \wJ_{n}(\theta)$ decreases on $[-1,0]$ for any $n\ge 3$.
\end{Prop}
  
\proof

By \eqref{eq:main_theorem_2}, it is enough to prove that $\theta\mapsto \tp_{n}^{\,U}(\theta)$ is increasing on $(-\infty,1]$ for all $n\ge 2$. Setting $r = -1/\theta\in (0,1]$, we see by the definition that $\tp_{n}^{\,U}(\theta)$ is, for every $n\ge 2$, the volume of the polytope in $[0,1]^{n}$ defined by the inequalities
\begin{equation*}
   x_{i}\, +\, \theta x_{i-1}\, +\, \cdots\, +\, \theta^{i-1} x_{1}\; \in\; [0,r]
\end{equation*}
for $i=1,\ldots, n-1$, and $x_{n}\ge x_{n-1}\, +\, \theta x_{n-2} \, + \, \cdots\, +\, \theta^{n-1} x_{1}$. Changing the variable $y_{1} = x_{1}$ and $y_{i}= x_{i}\, +\, \cdots\, +\, \theta^{i-1} x_{1}$ for $i = 2,\ldots,n$, we see that this volume equals that of the intersection of $[0,r]^{n-1}\times[0,1]$ and the polytope in $[0,1]^{n}$ defined by the inequalities
\begin{equation*}
   y_{i+1}\; \le\; 1\, +\, \theta x_{i-1}
\end{equation*}
for $i =1,\ldots, n-1$, which implies that $\theta\mapsto \tp_{n}^{\,U}(\theta)$ increases on $(-\infty,1]$.   
\qed

\begin{Rem}\label{Inc3}\rm
When combined with Remark \ref{Inc1}, the above proof shows that the mapping $\theta\mapsto p_{n}^{U}(\theta)$ is nondecreasing on the whole real line for all $n\ge 0$. Since  this is actually true on $\R^{+}$ for arbitrary innovation law, see again Remark \ref{Inc1}, it would be interesting to know if there are innovation laws where monotonicity fails to hold on the negative halfline.
\end{Rem}

\subsection{Asymptotic behaviour}

The following result gives exact exponential behavior of the $p^{\,U}_{n}(\theta)$ for $\theta<-1$, with a rate expressed in terms of the first root $z_{1/\theta}$ of the generalized expo\-nential function $E(1/\theta,z)$, similar to Proposition \ref{Asymp1}.
 
\begin{Prop}
\label{Asymp2}
For every $\theta <-1$, there exists $c_{\theta} \in (0,\infty) $ such that  
\begin{equation*}
    p^{\,U}_{n}(\theta)\; \sim \; \frac{1}{c_{\theta}^{}\,\mu_{\theta}^{n}}\qquad\text{as $n\to\infty$,}
\end{equation*}
where $\mu_{\theta} = 2(1-\theta) z_{1/\theta} > -2\theta$.
\end{Prop}

\begin{proof}
Putting $r= 1/\theta\in (-1,0)$ and adopting the notation of Proposition~\ref{Asymp1}, Eq.\,\eqref{AltExp} provides us with
\begin{equation*}
   G(r,z)\,:=\,\sum_{n\ge 0} (1-r)^{n}\,\wJ_{n+1}(r)\, \frac{z^{n}}{n!}\; =\; \frac{E(r,z)}{E(r, rz)}\; =\; \frac{{\displaystyle \prod_{k\ge 1} \lpa 1 + \frac{z}{a_{k}(r)}\rpa}}{{\displaystyle \prod_{k\ge 1} \lpa 1 + \frac{r z}{a_{k}(r)}\rpa}} 
\end{equation*}
for any $z\in \cB(0, b_{1}(r))$, where $b_{1}(r) = -\theta a_{1}(r) > 0$ equals the smallest root in modulus of $z\mapsto E(r,r z)$. We claim that $E(r, b_{1}(r))\neq 0$, for otherwise we would infer $E(r, -a_{1}(r)) = E(r,b_{1}(r)=0$ and hence, by Rolle's theorem, $E'(r,z_{1})=E(r,r z_{1})=0$ for some $z_{1}\in (-a_{1}(r), b_{1}(r))$, which contradicts the minimality of $b_{1}(r)$. Recalling from the proof of Proposition \ref{Asymp1} that $b_2(r) = -\theta\vert a_2(r)\vert > b_{1}(r)$, we see that $z\mapsto G(r,z)$ is meromorphic on $\cB(0,b_2(r))$ with a single pole of order one at $b_{1}(r)$. It follows upon setting
\begin{gather*}
c_{\theta}\; =\; \frac{{\displaystyle \prod_{k\ge 2} \lpa 1 + \frac{\theta b_{1}(r)}{a_{k}(r)}\rpa}}{{\displaystyle \prod_{k\ge 1} \lpa 1 + \frac{b_{1}(r)}{b_{k}(r)}\rpa}}\; \in\; (0,\infty),
\shortintertext{that the function}
z\;\mapsto\; G(r,z)\; -\; \frac{b_{1}(r)}{c_{\theta}(b_{1}(r) - z)}
\end{gather*}
is holomorphic on $\cB(0,b_2(r))$ and, by evaluating its coefficients, we finally obtain
\begin{equation*}
   p_{n}^{U}(\theta)\, \mu_{\theta}^{n}\; =\; (1-r)^{n}\,\wJ_{n+1}(r)\, \frac{(b_{1}(r))^{n}}{n!}\;\to\; \frac{1}{c_{\theta}}\qquad\text{as $n\to\infty.$}
\end{equation*}
Finally, the bound $\mu_{\theta} = -\theta \lambda_r > -2\theta$ follows by a look at the dual case $\theta\in (-1,0]$ in Proposition \ref{Asymp1}, see Remark \ref{Conv}(b). \qed
\end{proof}

\begin{Rem}\rm
As expected, one has $\mu_{\theta} = -\theta\lambda_{1/\theta} > \lambda_{1/\theta}$ and $\mu_{\theta} \to \infty$ as $\theta\to -\infty$. And since $\theta\mapsto p_{n}^{U}(\theta)$ is nondecreasing on $(-\infty,-1)$, the same must hold for $\theta\mapsto \mu_{\theta}^{-1}$ on this interval.
\end{Rem}

\section{The case $\theta>0$}

\subsection{Proof of Theorem~\ref{thm:main_3}}
\label{sec:case_>1}

Again, the argument relies on a linear recurrence relation obtained from the general formula \eqref{MecGen}. The first step is to show that
\begin{equation}\label{eq:main_rec, theta<=-1}
p_{n}(\theta)\; =\; \sum_{k=1}^{n}p_{n-k}(\theta)\lpa p_{k-1}(1/\theta)\, -\, p_{k}(1/\theta)\rpa
\end{equation}
for all $\theta > 0 $ and $n\ge 1$. Setting $r =1/\theta > 0$, we embark on the fact that
\begin{equation*}
    p_{n}(\theta)\; =\; \Prob\Bigg[\sum_{j=1}^{k} r^{j-1}X_{j}\ge 0, \; k=1,\ldots, n\Bigg].
\end{equation*}
For $n\ge 1$ and $u\in\R$, we define
\begin{equation*}
   A_{n}(u)\; =\; \Bigg\{\sum_{j=1}^{k} r^{j-1}X_{j}\ge 0,\; k=1,\ldots, n-1;\ \sum_{j=1}^{n} r^{j-1}X_{j}\ge u\Bigg\}
\end{equation*}
and $A_{n}=A_{n}(0)$. Since 
\begin{equation}\label{BigStrawHat}
A_{n}(u) \; =\; A_{n-1} \; \cap\; \Bigg\{\sum_{j=1}^{n-1} r^{j-1}X_{j}\ge u \, -\, r^{n-1} X_{n}\Bigg\}
\end{equation}
and the innovation law $F$ is symmetric and continuous, the decomposition
\begin{align*}
p_{n}(\theta)\;=\; \Prob[A_{n}]\;&=\; \frac{p_{n-1}(\theta)}{2}\,+\,\int_{-\infty}^{0}\Prob\left[A_{n-1}(-r^{n-1} u_{1})\right] dF(u_{1})\\
\;& =\;  p_{n-1}(\theta)\,q_{1}(r)\;+\;\int^{\infty}_{0}\Prob\left[A_{n-1}(r^{n-1} u_{1})\right] dF(u_{1})
\end{align*}
holds, where $q_{k}(r) = p_{k-1}(r) - p_{k}(r)$ for all $k\ge 1$. The next step is to evaluate the integrand, say $I_{n-1}$, in the previous line:
\begin{align*}
I_{n-1}\; & =\; \int_{ru_{1}}^{\infty} p_{n-2}(\theta)\, dF(u_2)\\
& \qquad\qquad\quad +\;\int^{ru_{1}}_{-\infty}\Prob\left[A_{n-2}(r^{n-1} u_{1} - r^{n-2} u_2)\right] dF(u_2)\\
& =\;\; p_{n-2}(\theta)\lpa 1 \, -\, \int^{ru_{1}}_{-\infty} dF(u_2)\rpa\\
& \qquad\qquad\quad +\;\int_{-ru_{1}}^{\infty}\Prob\left[A_{n-2}(r^{n-1} u_{1} + r^{n-2} u_2)\right] dF(u_2)\\
& =\;\; p_{n-2}(\theta)\lpa 1 \, -\, \int_{-ru_{1}}^{\infty} dF(u_2)\rpa\\
& \qquad\qquad\quad +\;\int_{-ru_{1}}^{\infty}\Prob\left[A_{n-2}(r^{n-1} u_{1} + r^{n-2} u_2)\right] dF(u_2).
\end{align*}
Putting things together, we arrive at
\begin{align*}
p_{n}(\theta) \; & =\; p_{n-1}(\theta)\,q_{1}(r)\; +\; p_{n-2}(\theta)\lpa p_{1}(r) \, -\, \int_{0}^{\infty}\!\!\! \int_{-ru_{1}}^{\infty} dF(u_{1}) dF(u_2)\rpa\\
& \qquad\qquad +\;\int_{0}^{\infty}\!\! \int_{-ru_{1}}^{\infty}\Prob\left[A_{n-2}(r^{n-1} u_{1} + r^{n-2} u_2)\right] dF(u_{1}) dF(u_2)\\
\; & =\; p_{n-1}(\theta)\,q_{1}(r)\; +\; p_{n-2}(\theta)\, q_2(r)\\
& \qquad\qquad +\;\int_{0}^{\infty}\!\! \int_{-ru_{1}}^{\infty}\Prob\left[A_{n-2}(r^{n-1} u_{1} + r^{n-2} u_2)\right] dF(u_{1}) dF(u_2),
\end{align*}
and calls for a computation of $I_{n-1}:=\Prob[A_{n-2}(r^{n-1} u_{1} + r^{n-2} u_2)]$ as the next step. Introducing the notation $s_{k}= r^{n-k}u_{k}+\ldots+r^{n-1}u_{1}$ for $k=1,\ldots n$, we find
\begin{multline*}
I_{n-2}\ =\; \int_{ru_2 + r^{2}u_{1}}^{\infty} \!\!\!\!\! p_{n-3}(\theta)\, dF(u_3)\ +\; \int^{\infty}_{-(ru_2 + r^{2}u_{1})}\!\!\! \Prob [A_{n-3}(s_3)]\ dF(u_3)\\
=\; p_{n-3}(\theta)\lpa 1 \, -\, \int_{-(ru_2 + r^{2}u_{1})}^{\infty} dF(u_3)\rpa+\int^{\infty}_{-(ru_2 + r^{2}u_{1})}\!\!\!\Prob [A_{n-3}(s_3)]\, dF(u_3)
\end{multline*}
and thereupon
\begin{multline*}
p_{n}(\theta) \;  =\; \sum_{k=1}^{3}p_{n-k}(\theta)\, q_{k}(r) \\
 \qquad +\;\int_{0}^{\infty}\!\!\! \int_{-ru_{1}}^{\infty}\! \int^{\infty}_{-(ru_2 + r^{2}u_{1})}\!\!\!\!\!\Prob [A_{n-3}(s_3)]\, dF(u_{1}) dF(u_2) dF(u_3).
\end{multline*}
Therefore, by continuing in the now obvious manner and repeatedly using \eqref{BigStrawHat}, we finally arrive at 
\begin{equation*}
    p_{n}(\theta) \; =\; \sum_{k=1}^{n}p_{n-k}(\theta)\, q_{k}(r) 
\end{equation*}
as required for \eqref{eq:main_rec, theta<=-1}. But this identity in combination with $p_{0}(1/\theta) = 1$ implies 
\begin{equation*}
   \sum_{k=0}^{n}p_{n-k}(\theta)\, p_{k}(1/\theta) \; = \; \sum_{k=1}^{n}p_{n-k}(\theta)\, p_{k-1}(1/\theta) \; = \; \sum_{k=0}^{n-1}p_{n-1-k}(\theta)\, p_{k}(1/\theta)
\end{equation*}
for each $n\ge 1$ and $\theta > 0$, that is, the quantity on the very left does not depend on $n$ and must therefore equal $1$ as claimed.\qed

\begin{Rem}\label{MRGL}\rm
(a) In the case when $X_{1}$ has a biexponential law and $\theta \in (0,1)$, the generating function of the persistence probabilities can be computed from the results of \cite{Larr04} in terms of $q$-series. More precisely, it follows from Formula (31) in \cite{Larr04} that 
\begin{equation}
\label{tet}
\sum_{n\ge 0}\, p_{n}(\theta)\, z^{n}\; = \;\frac{(\theta z; \theta^{2})_\infty + (\theta^{2} z; \theta^{2})_\infty}{ (z; \theta^{2})_\infty + (\theta z; \theta^{2})_\infty}
\end{equation} 
with the standard $q$-notation $(z;q)_\infty = \prod_{n\ge 0} (1- zq^{n})$ for all $z,q\in (0,1)$. Thanks to Theorem~\ref{thm:main_3}, we can now also compute this generating function if $\theta > 1$, namely
\begin{equation}
\label{thet}
\sum_{n\ge 0}\, p_{n}(\theta)\, z^{n}\; = \;\frac{(z; \theta^{-2})_\infty + (\theta^{-1} z; \theta^{-2})_\infty}{(1-z)\, \lpa(\theta^{-1} z; \theta^{-2})_\infty + (\theta^{-2} z; \theta^{-2})_\infty\rpa}.
\end{equation}
These formulae for $p_{n}(\theta)$ if $\theta > 0$ are significantly more complicated than for $\theta < 0$, see Remark \ref{Expo}. On the other hand, they exhibit some interesting similarities with those in the case of uniform innovations, due to the $q$-binomial theorem. Skipping details, formula \eqref{tet} can indeed be rewritten
\begin{equation*}
   \sum_{n\ge 0}\, p_{n}(\theta)\, z^{n}\; = \;\frac{\displaystyle\sum_{n\ge 0} \frac{q^{n(n+1)/2}}{(q-1)^{n}}\, \frac{z^{n} (1+ \theta^{-n})}{[n]_q!}}{\displaystyle\sum_{n\ge 0}\frac{q^{n(n-1)/2}}{(q-1)^{n}}\, \frac{z^{n}(1+\theta^{n})}{[n]_q!}}
\end{equation*}
with $q = \theta^{2}\in (0,1)$, and resembles the formula
\begin{equation*}
   \sum_{n\ge 0}\, p_{n}^{U}(\theta)\, z^{n}\; = \;\frac{\displaystyle\sum_{n\ge 0} \frac{\theta^{n(n+1)/2}}{(\theta-1)^{n}}\, \frac{z^{n}}{2^{n}\, n!}}{\displaystyle\sum_{n\ge 0}\frac{\theta^{n(n-1)/2}}{(\theta-1)^{n}}\, \frac{z^{n}}{2^{n}\, n!}}, 
\end{equation*}
which in turn follows from the fact that expressions in \eqref{eq:to_prove-MR} and \eqref{eq:to_prove-0} are identical for $\theta\in [-1,\frac{1}{2}]$. This raises the question whether other symmetric innovation laws exist, intermediate between uniform and bi-exponential, that allow one to give the explicit persistence probabilities in Theorems \ref{thm:main_2} and \ref{thm:main_3}.

\vspace{.2cm}
(b) Recall from Remark \ref{MRGenLin}(a) that in the case when the innovation law is uniform on $[-1,1]$, Formula \eqref{eq:main_theorem_3} for the persistence probabilities $p_{n}^{U}(\theta)$ is valid for each $n\ge 1$ and $\theta\ge 1/\theta_{n}$, where $\theta_{n}$ denotes the positive solution to $\theta +\cdots +\theta^{n-1} = 1$. In the case $\theta\in(1,1/\theta_{n})$, the behavior of $p_{n}^{U}(\theta)$ exhibits a more exotic character (truncated Laurent series in $\theta$). For example, one has 
\begin{equation*}
   \tp_3^{\,U}(\theta) \; =\; - \frac{1}{\theta}\, +\, \frac{19}{6}\, +\, \frac{\theta^{2}}{2}\, -\, \frac{\theta^3}{6}
\end{equation*}
for each $\theta \ge 1$ such that $\theta^{-1} +\theta^{-2} \ge 1$. Still, the mapping $\theta\to\tp_{n}(\theta)$ seems to maintain a certain degree of smoothness on $(1,\infty)$ for all $n$, see Figure~\ref{fig:plot_2345}.
\end{Rem}

\subsection{Proof of Corollary \ref{cor:Thm 3} and properties of the $\hJ_{n}(\theta)$}
\label{SphJ}

Regarding a proof of Corollary \ref{cor:Thm 3}, we first mention that Eq.\,\eqref{eq:main_theorem_3} for the persistence probabilities $p_{n}^{\,U}(\theta)$ follows directly from Theorems \ref{thm:main_1} and \ref{thm:main_3} when defining the $\hJ_{n}(\theta)$ by \eqref{eq:modified_MR_bis}. Therefore, it remains to verify the asserted properties of the latter functions, which is done by Proposition \ref{Mischba} below, followed by a discussion of some further notable properties.

\begin{Prop}\label{Mischba}
For each $n\ge 1$, $\hJ_{n+1}(\theta)$ is a polynomial in $\Z[X]$ of degree $n(n-1)/2$ that has coefficient $2^{n-1} n!$ of order 0 while all other coefficients are negative.
\end{Prop}

\begin{proof}
It follows from \eqref{eq:modified_MR} and \eqref{eq:modified_MR_bis} that the $\hJ_{n}(\theta)$ and $\wJ_{n}(\theta)$ are related through their exponential generating functions, namely
\begin{gather}
\sum_{n\ge 0} \hJ_{n+1}(\theta)\, \frac{z^{n}}{n!}\; =\; \frac{1}{1-2z}\lpa\sum_{n\ge 0} (-1)^{n} \wJ_{n+1}(\theta)\, \frac{z^{n}}{n!}\rpa,\nonumber
\shortintertext{which is equivalent to}
\frac{\hJ_{n+1}(\theta)}{2^{n}\,n!}\; =\; \sum_{k = 0}^{n}  \, \frac{(-1)^k \wJ_{k+1} (\theta)}{2^k\,k!}\label{hwJ}
\end{gather}
for all $n\ge 0$ and implies the recursive relation
\begin{equation}
\label{hwJ1}
\hJ_{n+1}(\theta)\; =\; 2n\,\hJ_{n}(\theta)\, +\, (-1)^{n}\wJ_{n+1}(\theta)
\end{equation}
for all $n\ge 1$, with initial condition $\hJ_{1}(\theta)=1$. As an immediate consequence of this relation and formally proved by induction, each $\hJ_{n+1}(\theta)$ is indeed a polynomial in $\Z[X]$ of the asserted degree. We further infer from \eqref{hwJ} that
\begin{equation*}
   \hJ_{n+1}(\theta)\; =\; 2^{n-1} n!\; -\; 2^{n} n!\,\sum_{k = 2}^{n}\,\frac{(-1)^{k+1} \wJ_{k+1} (\theta)}{2^k\,k!} 
\end{equation*}
for any $n\ge 2$, which together with $\hJ_{2}(\theta)=1$ concludes the proof because, as a consequence of Proposition \ref{Misc}, the polynomial
\begin{equation*}
   \sum_{k = 2}^{n}\frac{(-1)^{k+1} \wJ_{k+1}(\theta)}{2^k\,k!} 
\end{equation*}
has valuation $1$ and positive coefficients for any $n\ge 2$.\qed
\end{proof}

For $\theta \ge 2$, Formula \eqref{hwJ} implies the following curious invariance property in the expansion of the persistence probability $p_{n}^{U}(\theta)$. 

\begin{Prop}\label{MischMisch}
For any $\theta\ge 2$, $k\ge 0$ and $n\ge k+1$, the first $k$ terms in the polynomial expansion of $p_{n}^{U}(\theta)$ as a function of $1/\theta$ do not depend on $n.$
\end{Prop}

\begin{proof}
It follows by Theorem \ref{thm:main_3} and \eqref{hwJ} that
\begin{align}
p^{\,U}_{n}(\theta)\; =\; \frac{\hJ_{n+1}(1/\theta)}{2^{n}\,n!}\ &=\ \sum_{j= 0}^{n}\,\frac{(-1)^{j} \wJ_{j+1}(1/\theta)}{2^{j} j!}\nonumber\\
&=\ \sum_{j=0}^{k+1}\,\frac{(-1)^{j}\wJ_{j+1} (1/\theta)}{2^{j} j!} \; +\; P_{n,k}(1/\theta)
\label{eq:constant terms}
\end{align}
for any $n\ge k+1$, where $P_{n,k}(1/\theta)$ is zero for $n=k+1$ and, by Proposition~\ref{Misc}, a polynomial of valuation $k+1$ for $n > k+1$. This shows that the $k$ first terms of the expansion of $p_{n}^{U}(\theta)$ in $1/\theta$ are those in the first sum appearing in \eqref{eq:constant terms} and thus do not depend on $n\ge k+1$ as claimed.\qed
\end{proof}

We will see in the next subsection that $p^{\,U}_{n}(\theta)\downarrow\ell(\theta) > 0$ for all $\theta >1$. If $\theta\ge 2$, Proposition \ref{MischMisch} contributes to this convergence result  an infinite series representation of the limit, namely
\begin{equation*}
   \ell(\theta)\; =\; \sum_{k\ge 0} a_{k}\,\theta^{-k}, 
\end{equation*}
where $a_{0}=\frac{1}{2}$ and $a_{k}$ for $k\ge 1$ equals the coefficient of $\theta^{k}$ in the polynomial sum
\begin{equation*}
   \sum_{j=0}^{k+1}\,\frac{(-1)^{j}\wJ_{j+1} (\theta)}{2^{j}\,j!} 
\end{equation*}
and is a negative rational. The first terms in the expansion of $\ell(\theta)$ are 
\begin{multline*}
\ell(\theta)\;=\; \frac12\,-\,\bigg(\frac{1}{8\,\theta}\,+\, \frac{1}{16\,\theta^{2}}\, +\, \frac{5}{96\,\theta^3}\, +\, \frac{1}{24\,\theta^4} \, +\, \frac{5}{128\,\theta^5}\\
+\, \frac{7}{192\,\theta^6}\, +\, \frac{9}{256\,\theta^7}\, +\, \frac{107}{3072\,\theta^8}\, +\, \frac{641}{18432\,\theta^9}\,+\, \cdots\bigg)
\end{multline*}
and exhibit some curious combinatorial behavior. For example, the sequence of ratios $\{a_{k+1}/a_{k},\, k\ge 1\}$ begins with   
\begin{equation*}
   \frac{1}{2},\ \frac{5}{6},\ \frac{4}{5},\ \frac{15}{16}, \ \frac{14}{15},\ \frac{27}{28}, \ \frac{107}{108},\ \frac{641}{642},\, \ldots
\end{equation*}
and thus seems to contain terms of the form $\frac{N}{N+1}$ only for some integers $N$. A thorougher investigation of the limit function $\ell(\theta)$ and its coefficients $a_{k}$ for $k\ge 1$ is left open for future research.

\vspace{.2cm}
Recalling that $T_{\theta}^{\,U}=\inf\{n:Y_{n}<0\}$ in the case of uniform innovations on $[-1,1]$,
the last result of this subsection provides a surprisingly simple formula for the first hitting probabilities
$\Prob[T_{\theta}^{\,U} = n]$ in the case $\theta\ge 2.$

\begin{Prop}
\label{Hit}
If $\theta\ge 2$, then
\begin{equation*}
    \Prob[T_{\theta}^{\,U} = n]\; =\; \frac{(-1)^{n-1} \wJ_{n+1}(1/\theta)}{2^{n} \, n!} 
\end{equation*}
holds for all $n\ge 1$.
\end{Prop} 

\begin{proof}
This follows from
\begin{align*}
\Prob[T_{\theta}^{\,U} = n]\; &=\;p_{n-1}^{\,U}(\theta)\,-\,p_{n}^{U}(\theta)\\
&=\ \frac{2n \hJ_{n}(1/\theta) \, -\, \hJ_{n+1}(1/\theta)}{2^{n}\, n!}\ =\ \frac{(-1)^{n-1}\wJ_{n+1}(1/\theta)}{2^{n}\,n!}
\end{align*}
for every $n\ge 1$, where the second equality holds by Theorem~\ref{thm:main_3} and the third one by \eqref{hwJ1}.\qed
\end{proof}

We note as a particular consequence that the mapping $\theta\mapsto \Prob[T_{\theta}^{\,U} = n]$ is decreasing on $[2,\infty)$ for all $n\ge 1$. This follows because the polynomials $(-1)^{n-1} \wJ_{n+1}(\theta)$ have only positive coefficients, and it refines the assertion that $\theta\mapsto 1- \ell(\theta) = \sum_{n\ge 1}\Prob[T_{\theta}^{\,U} = n]$ decreases in $\theta$ which is clear from the above series expansion for $\ell(\theta).$
 
\subsection{Asymptotic behaviour}
\label{Asymp3}

The next result is a further consequence of Theorem \ref{thm:main_3} and provides an unexpected extension of a well-known result for dual pairs of ladder epochs of  ordinary random walks (which occurs here if $\theta=1$) to general AR(1) processes with drift $\theta>1$ and symmetric, continuous innovation law. To see the connection with ladder epochs, we point out the obvious facts that
\begin{equation*}
   p_{n}(\theta)\ =\ \Prob[T_{\theta}>n]\ \to\ \Prob[T_{\theta}=\infty]\qquad\text{as }n\to\infty, 
\end{equation*}
thus $\ell(\theta)=\Prob[T_{\theta}=\infty]$, and that $T_{\theta}=\inf\{n:Y_{n}<0\}$ can be viewed as the first descending ladder epoch of $(Y_{n})_{n\ge 0}$.
 
\begin{Prop} \label{Hout}
Given the assumptions of Theorem \ref{thm:main_3}, the identity
\begin{equation}\label{eq:ladder epoch id}
\ell(\theta)\ =\ \Prob[T_{\theta}=\infty]\ =\ \frac{1}{\Erw[T_{1/\theta}]}
\end{equation}
holds true, and the terms are positive if and only if  $\theta>1$ and $\Erw\log (1+\vert X_{1}\vert)<\infty$.
\end{Prop}

If $\theta=1$, then $(Y_{n})_{n\ge 0}$ is an ordinary random walk with continuous and symmetric increment law. Denoting by $T_{1}^{*}=\inf\{n:Y_{n}\le 0\}$ its first weakly descending ladder epoch, it follows from the famous Spitzer-Baxter identities, see e.g.~\cite[Sect.\,8.4]{Chung:01}, that $(T_{1},T_{1}^{*})$ form a dual pair satisfying
\begin{equation}\label{eq:dual pair id}
\Prob[T_{1}=\infty]\ =\ \frac{1}{\Erw T_{1}^{*}}\quad\text{and}\quad\Prob[T_{1}^{*}=\infty]\ =\ \frac{1}{\Erw T_{1}}.
\end{equation}
But under the given additional assumptions on the increment law, it follows immediately that $T_{1}$ and $T_{1}^{*}$ are identically distributed, so that \eqref{eq:ladder epoch id} can indeed be viewed as an extension of \eqref{eq:dual pair id} of the aforementioned kind. Finally, we should note that all quantities in \eqref{eq:dual pair id} are $0$ because the random walk is symmetric and thus particularly oscillating.

\begin{proof}[of Proposition \ref{Hout}]
We already pointed out in the Introduction that an AR(1) process $(Y_{n})_{n\ge 0}$ defined by \eqref{eq:AR(1)} is positive recurrent if and only if  $\theta\in (-1,1)$ and $\Erw\log (1+\vert X_{1}\vert)<\infty$. In particular, we infer here that $\theta>1$ is necessary and sufficient for
\begin{equation*}
\Erw T_{1/\theta}\ =\ \sum_{n\ge 0}\Prob[T_{1/\theta}>n]\ =\ \sum_{n\ge 0}p_{n}(1/\theta)\ <\ \infty.
\end{equation*}
Now use \eqref{eq:duality theta>0} in Theorem \ref{thm:main_3} to infer
\begin{equation*}
1\ =\ \sum_{k=0}^{n}\Prob[T_{\theta}>k]\,\Prob[T_{1/\theta}>n-k]\ \ge\ \Prob[T_{\theta}=\infty]\sum_{k=0}^{n}\Prob[T_{1/\theta}>n-k]
\end{equation*}
for each $n\ge 0$ and then upon letting $n$ tend to $\infty$ that
\begin{equation*}
   \Prob[T_{\theta}=\infty]\,\Erw T_{1/\theta}\ \le\ 1, 
\end{equation*}
in particular that $\Prob[T_{\theta}=\infty]>0$ entails $\Erw T_{1/\theta}<\infty$ and so $\theta>1$. Conversely, if $\Erw T_{1/\theta}<\infty$, then \eqref{eq:duality theta>0} provides, for any $m\ge 0$ and $n>m$
\begin{align*}
1\ \le\ \Prob[T_{\theta}=\infty]\sum_{k=0}^{m}\Prob[T_{1/\theta}>k]\ +\ \sum_{k=m+1}^{n}\Prob[T_{1/\theta}>k]
\end{align*}
and then upon letting $n$ tend to $\infty$ that
\begin{align*}
1\ \le\ \Prob[T_{\theta}=\infty]\sum_{k=0}^{m}\Prob[T_{1/\theta}>k]\ +\ \sum_{k>m}\Prob[T_{1/\theta}>k].
\end{align*}
By finally taking the limit $m\to\infty$, we arrive at
\begin{equation*}
   1\ \le\ \Prob[T_{\theta}=\infty]\,\Erw T_{1/\theta},
\end{equation*}
which together with the first part proves the equivalence of $\Prob[T_{\theta}=\infty]>0$ and $\Erw T_{1/\theta}<\infty$ (and thus $\theta>1$) as well as Eq.\,\eqref{eq:ladder epoch id}.\qed
\end{proof}

\begin{Rem}\rm
If $\theta>1$ and for biexponential innovations, a combination of Proposition \ref{Hout} and Formula (34) in \cite{Larr04} implies
\begin{equation*}
   \ell(\theta)\ =\ \frac{(\theta^{-1}; \theta^{-2})_\infty}{(\theta^{-1}; \theta^{-2})_\infty +\, (\theta^{-2}; \theta^{-2})_\infty}\; >\; 0 
\end{equation*}
which can also be seen directly from \eqref{thet} and the Hardy-Littlewood Tauberian theorem.
\end{Rem}

\begin{Rem}\rm
Back to the case when innovations are uniform on $[-1,1]$, Proposition \ref{Hout} provides   
\begin{equation}\label{Limes}
\ell(\theta)\; =\; \frac{1}{{\displaystyle \sum_{n\ge 0} p_{n}^{U}(1/\theta)}}\; >\; 0\qquad\text{as $n\to\infty$,}
\end{equation}
for any $\theta > 1$. This should be compared to the following: for $\theta\ge 2$, we have already given in the previous subsection an alternative formula for $\ell(\theta)$ in terms of the sequence of modified Mallows-Riordan polynomials, namely
\begin{equation*}
   \ell(\theta)\; =\; \sum_{n\ge 0}\, \frac{(-1)^{n} \wJ_{n+1}(1/\theta) }{2^{n}\, n!}\; =\; \sum_{k\ge 0} a_{k}\, \theta^{-k}, 
\end{equation*}
with a somewhat mysterious sequence of coefficients $a_{k}$ in the convergent series representation on the right-hand side.
\end{Rem}

The final result of this subsection confirms that $p_{n}^{U}(\theta)$ for $\theta>1$ approaches its limit $\ell(\theta)$ again at an exponential rate, in the case $\theta\ge 2$ given by the first positive root
\begin{gather*}
\nu_{\theta}\ =\ \inf\{ z > 0: \; L(\theta,z) = 0\}
\shortintertext{of the function}
L(\theta,z)\; =\; \frac{1}{a_{1}(1/\theta)}\; +\; \sum_{k\ge 2} \frac{{\displaystyle 1 - z/\lambda_{1}(\theta)}}{{\displaystyle a_{k}(1/\theta) \lpa 1 - z/\lambda_{k}(\theta) \rpa}}.
\end{gather*}
Here we have used the notation from Proposition \ref{Asymp1} and Remark \ref{Conv}(a), and further defined $\lambda_{k}(\theta) = 2(1-1/\theta)a_{k}(1/\theta)$. It should be recalled that $(a_{k}(1/\theta))_{k\ge 1}$ and $(\lambda_{k}(\theta))_{k\ge 1}$ are increasing sequences of positive numbers with $\lambda_{1}(\theta) > 1$ and 
\begin{equation*}
   \sum_{k\ge 1} \frac{1}{a_{k}(1/\theta)}\; = \; 1\; < \;\infty. 
\end{equation*}
This implies that $L(\theta,z)$ is real-analytic in $z$ on $[0,\lambda_2(\theta))$, where it decreases from 1 to $-\infty$ and has a unique and simple root $\nu_{\theta}\in (\lambda_{1}(\theta), \lambda_2(\theta))$. 

\medskip

\begin{Prop}
\label{ExpQ}
For any $\theta > 1$, there exists $\kappa_{\theta} > 0$ such that
\begin{equation}\label{eq:exp bound theta>1}
p_{n}^{U}(\theta)\, -\, \ell(\theta) \; \le\; e^{-\kappa_{\theta} n}\qquad\text{for all }n\ge 0,
\end{equation}
and if $\theta \ge 2$, there further exists $c_{\theta} > 0$ such that
\begin{equation}\label{eq:exp rate theta>1}
p_{n}^{U}(\theta)\, -\, \ell(\theta)\; \sim \; \frac{1}{c_{\theta}\,\nu_{\theta}^{n}}\qquad \text{as }n\to\infty.
\end{equation}
\end{Prop}

\begin{proof}
Putting $r_{n}^{\,U}(\theta) = p_{n}^{U}(\theta) -\ell(\theta) > 0$, it follows from \eqref{Limes} and Theorem~\ref{thm:main_3} that
\begin{equation*}
   \Bigg(\sum_{n\ge 0} r_{n}^{\,U}(\theta)\, z^{n}\Bigg)\Bigg(\sum_{n\ge 0} p_{n}^{U}(1/\theta)\, z^{n}\Bigg)\; =\;  \ell(\theta)\,\sum_{n\ge 0}p_{n}^{U}(1/\theta)\, (1+\cdots +z^{n-1}) 
\end{equation*}
for any $z\in (0,1)$ and $\theta > 1$. Moreover, the right-hand side has an analytic extension to $(0, z_*(\theta))$ for some $z_*(\theta) > 1$ satisfying $p_{n}^{U}(1/\theta) z_*(\theta)^{n}\to 0 $ as $n\to\infty$, see e.g.\ Theorem 1 in \cite{KN08} for the latter property. This clearly implies \eqref{eq:exp bound theta>1}.

\vspace{.1cm}
For $\theta \ge 2$, we have seen in Remark \ref{Conv}(a) that, with the above notation,
\begin{equation*}
   p^{\,U}_{n}(1/\theta)\; =\; \sum_{k\ge 1} \frac{1}{a_{k}(1/\theta)\,\lambda_{k}(\theta)^{n}} 
\end{equation*}
for every $n\ge 0$. By plugging this in the above equation, we find after some easy simplifications and upon using Fubini's theorem that
\begin{equation*}
   \sum_{n\ge 0} r_{n}^{\,U}(\theta)\, z^{n}\; =\; \frac{\ell(\theta)\,{\wtil L}(\theta,z)}{L(\theta,z)} 
\end{equation*}
for any $z\in [0,\lambda_{1}(\theta))$, with the notation ${\wtil a_{k}}(1/\theta) = (\lambda_{k}(\theta) -1) a_{k}(1/\theta)$ for $k\ge 1$ and 
\begin{equation*}
   {\wtil L}(\theta,z)\; =\; \frac{1}{{\wtil a_{1}}(1/\theta)}\; +\; \sum_{k\ge 2} \; \frac{{\displaystyle 1 - z/\lambda_{1}(\theta)}}{{\displaystyle {\wtil a_{k}}(1/\theta) \lpa 1 - z/\lambda_{k}(\theta) \rpa}}\cdot
\end{equation*}
As $(\lambda_{k}(\theta))_{k\ge 1}$ is increasing, the function $z\mapsto {\wtil L}(\theta,z)$, which is real-analytic and decreasing on $[0,\lambda_2(\theta))$ (as a sum of real-analytic and decreasing functions) must also satisfy ${\wtil L}(\theta,z) >0$. Putting everything together while skipping details, we finally obtain that the function
\begin{equation*}
   z\;\mapsto\;\sum_{n\ge 0} r_{n}^{\,U}(\theta)\, z^{n}
\end{equation*}
is meromorphic on $\cB(0,\lambda_2(\theta))$, where it has a unique and simple pole at $\nu_{\theta}$. This completes the proof as in Proposition \ref{Asymp1}.\qed
\end{proof}

\section{Miscellanea}
\label{sec:miscellaneous}

\subsection{The Tutte polynomial on a complete graph}\label{SoliTutti}

Given a finite graph $G$, its dichromatic polynomial is defined as the bivariate generating function
\begin{equation*}
   Q(G,x,\theta)\; =\; \sum_{H\subseteq G} x^{k(H)}\theta^{e(H) + k(H) -v(H)}, 
\end{equation*}
where summation ranges over all spanning subgraphs of $G$ and $e(H)$, $k(H)$, $v(H)$ denote the number of edges, connected components and vertices of $H$, respectively. For the complete graph $K_{n}$ with $n$ vertices, the exponential generating function of the $T_{n}(x,\theta) = Q(K_{n},x,\theta)$ has been computed in \cite[Eq.\,(17)]{Tutte67} as 
\begin{align}\label{ExpTutt}
\sum_{n\ge 0}\, T_{n}(x,\theta)\, \frac{z^{n}}{n!} &\; \fps\; \Bigg(\sum_{n \ge 0}\frac{(\theta +1)^{n(n-1)/2}}{\theta^{n}}\,\frac{z^{n}}{n!}\Bigg)^{\!\theta x}\nonumber\\
&\; \fps\; \exp\Bigg[x \sum_{n\ge 1} J_{n}(\theta +1)\, \frac{z^{n}}{n!}\Bigg],
\end{align}
where \eqref{id:MR} has been utilized for the second equality. Introducing the modified polynomials $T^{}_{K_{0}}(x,\theta) = 1$ and
\begin{equation*}
   T^{}_{K_{n}}(x,\theta)\; =\; 1\, +\, \frac{T_{n}(x-1,\theta-1)-1}{x-1} 
\end{equation*}
for $n\ge 1$, we have the identity
\begin{equation}\label{identification}
T^{}_{K_{n}}(1,\theta)\; =\; J_{n}(\theta).
\end{equation}
For any $q\in (0,1]$ and $t\ge 0$, consider now the Tutte polytope ${\bf T}_{n}(q,t)\subset \R^{n}$ defined as the set of $(x_{1},\ldots,x_{n})$ satisfying the constraints $x_{n} \ge 1-q$ and
\begin{equation*}
   qx_{j}\;\le\; q(1+t)x_{j-1}\,-\,t(1-q)(1-x_{i-1}) 
\end{equation*}
for all $1\le i\le j\le n$, with the convention $x_{0} = 1$. The limiting Tutte polytope is obtained as $q\to 0$ and equals 
\begin{equation*}
    {\bf T}_{n}(0,t)\; =\; \{ 1\le x_{1}\le 1+t, \; 1\le x_{i}\le (1+t) x_{i-1},\; i=2,\ldots, n\}. 
\end{equation*}
The following corollary follows directly from Proposition \ref{prop:MR_volume} and \eqref{identification} above.

\begin{Cor}\label{TuttiPoly}
For all $n\ge 1$ and $t\ge 0$, one has
\begin{equation*}
    {\rm vol}\,{\bf T}_{n}(0,t)\; =\; t^{n}\, T^{}_{K_{n+1}}(1,1+t)/n!. 
\end{equation*} 
\end{Cor}

The result particularly provides that the volume of the Cayley polytope
\begin{equation*}
   \bC_{n}\;=\;{\rm vol}\,{\bf T}_{n}(0,1)\ =\ \{1\le x_{1}\le 2, \; 1\le x_{i}\le 2 x_{i-1}, \; i= 2, \ldots, n\} 
\end{equation*}
equals $J_{n+1}(2)/n!$, that is $1/n!$ times the number of labeled connected graphs with $n+1$ vertices. This fact was conjectured in~\cite{BeBrLe13} and proved in~\cite{KoPa13}, as a consequence of the more general formula
\begin{equation}\label{eq:Tutte volume}
{\rm vol}\,{\bf T}_{n}(q,t)\; =\; t^{n}\, T^{}_{K_{n+1}}(1+q/t,1+t)/n!
\end{equation} 
for any $n\ge 1, q\in (0,1]$ and $t\ge 0$, see Theorem 1.3 therein. The limiting case $q=0$ follows directly from Proposition \ref{prop:MR_volume}, the proof of which is more elementary than the arguments developed in~\cite{KoPa13}. See also Theorem 3 in~\cite{KoPa14} for another elementary proof of Corollary \ref{TuttiPoly} in the case $t=1$ using partitions of integers. The natural question arises whether the general identity \eqref{eq:Tutte volume} admits a simple proof as well. 

\vspace{.2cm}
Before we finish this subsection by pointing out a curious connection between the Tutte polynomial $T_{n}(x,\theta)$ and a certain Poisson process on $\N$ for $\theta\in [-2,1)$, we prove the following summability criterion for Mallows-Riordan polynomials that seems to have been unnoticed in the literature.

\begin{Lemma}\label{Sum}
For $\theta\ge -1$, the positive series $\sum_{n\ge 1}\frac{J_{n}(\theta)}{n!}$ is finite if and only if  $\theta < 0.$
\end{Lemma}

\begin{proof}
The only-if part is immediate since $J_{n}(\theta)\ge J_{n}(0) = (n-1)!$ for all $n\ge 1$ and $\theta\ge 0$. For the if part, we first observe, with the notation of Subsection \ref{Asyb}, that for all $z\in (0,1)$ and $\theta\in[-1,0)$ one has by \eqref{id:MR}
\begin{equation*}
    (\theta-1) \log E\bigg(\theta, -\frac{z}{1-\theta}\bigg)\; =\; \sum_{n\ge 1}\; J_{n}(\theta)\, \frac{z^{n}}{n!}\; <\; \sum_{n\ge 1}\; \frac{z^{n}}{n}\; <\; \infty, 
\end{equation*}
recalling the fact mentioned in Remark \ref{Conv}(b) that $(1-\theta)z_{\theta} \ge 1$ for each $\theta\in [-1,0)$. Fixing now any such $\theta$ and recalling further that $\vth\mapsto(1-\vth)z_{\vth}$ is nonincreasing on $[-1,\frac{1}{2}]$ with value $z_{0} = 1$ at 0 shows that $(1-\theta)z_{\theta}=1$ entails $(1-\vth)z_{\vth} = 1$ for all $\vth\in(\theta,0)$, which in turn implies the impossible fact that the analytic function
\begin{equation*}
    \vth\;\mapsto \; \sum_{n\ge 0} \frac{\vth^{n(n-1)/2}}{(\vth-1)^{n}} 
\end{equation*}
vanishes on $(0,-\theta)$. So we must have $(1-\theta)z_{\theta} > 1$ for all $\theta\in[-1,0)$ and, by picking some $x_{\theta}\in (1,(1-\theta)z_{\theta})$, we finally obtain
\begin{equation*}
   \sum_{n\ge 1}\; \frac{J_{n}(\theta)}{n!}\; <\;\sum_{n\ge 1}\; J_{n}(\theta)\, \frac{x_{\theta}^{n}}{n!}\; <\; \infty
\end{equation*} 
as required.\qed
\end{proof}

\begin{Rem}\label{Asympb}\rm
Recall from Proposition \ref{Asymp1} that $\lambda_{\theta}=2(1-\theta)z_{\theta}$, thus  nonincreasing on $[-1,1)$, and from \eqref{AsympJn} in Remark \ref{Conv}(a) that
\begin{equation*}
    \frac{J_{n+1}(\theta)}{n!}\; \sim\; \frac{1}{z_{\theta}}\lpa\frac{2}{\lambda_{\theta}}\rpa^{n}\qquad\text{as }n\to\infty. 
\end{equation*}
As a consequence of the previous lemma, we now infer $\lambda_{\theta} > 2$ for $\theta \in [-1,0)$ and $\lambda_{\theta} < 2$ for $\theta \in (0,1).$
\end{Rem}

Lemma \ref{Sum} also ensures that the positive measure
\begin{equation*}
   m_{\theta}^{}(du)\; =\; \sum_{n\ge 1}\; \frac{J_{n}(\theta +1)}{n!}\;\delta_{n}^{}(du)
\end{equation*}
is finite with total mass ${\ovl m}_{\theta}^{} = \theta\log(E(\theta + 1,1/\theta))$ for any $\theta\in [-2,-1)$. Considering now a compound Poisson process $\{X_{\theta}(t), \; t\ge 0\}$ on $\N$ with L\'evy measure $m_{\theta}^{}(du)$, the discrete L\'evy-Khintchine formula, see e.g.\ Theorem~II.3.2 in \cite{STVH}, combined with \eqref{ExpTutt} implies
\begin{equation*}
   \sum_{n\ge 0}\, e^{-t{\ovl m}_{\theta}^{}}T_{n}(t,\theta)\, \frac{z^{n}}{n!}\; =\;  \exp\lcr -t\int_{0}^{\infty} (1-z^u) m_{\theta}^{} (du)\rcr\; =\; \Erw\lcr z^{X_{\theta}(t)}\rcr
\end{equation*}
for every $z\in (0,1)$. By comparison of coefficients, this leads to
\begin{equation*}
   \Prob\lcr X_{\theta}(t) = n\rcr\; =\; e^{-t{\ovl m}_{\theta}^{}}\, \frac{T_{n}(t,\theta)}{n!}
\end{equation*}
for all $t\ge 0$, $\theta\in[-2,1)$ and $n\ge 0$, and thus to a probabilistic representation of the Tutte polynomial. In the limiting case $\theta = -2$, ${\ovl m}_{-2}^{} = -\log(1-\sin 1)$ and the compound Poisson process has explicit moment generating function  
\begin{equation*}
   \Erw\lcr z^{X_{\!-\! 2}(t)}\rcr\; =\; \lpa\frac{1-\sin 1}{1 -\sin z}\rpa^t
\end{equation*}
for all $t\ge 0$ and $z\in (0,\pi/2).$
 
\subsection{Infinite divisibility}
\label{ID}

This subsection is devoted to aspects of infinite divisibility (ID) in connection with the persistence probabilities $p_{n}(\theta)$, and it begins with a discussion of the shifted first passage time below zero $\wtil{T}_{\theta}=T_{\theta}-1$, related to the $p_{n}(\theta)$ by
\begin{equation*}
   \Prob[\wtil{T}_{\theta}=n]\ =\ p_{n}(\theta)-p_{n+1}(\theta)\quad\text{for all }n\ge 0. 
\end{equation*}
Unlike $T_{\theta}$, it qualifies at all to have a discrete ID law by taking values in $\N_{0}$ rather than $\N$ only.

\vspace{.1cm}
For a downward skip-free Markov chain on $\Z$, it is easy to see by right continuity that its shifted passage time below zero is indeed always ID, but already replacing the state space with the whole real line makes the problem less immediate because of the jumps. Back to the model \eqref{eq:AR(1)} studied in this work and assuming $\theta=0$ (white-noise case), the random variable $\wtil{T}_{0}$ is geometric and hence ID regardless of the innovation law, see e.g.~Example II.2.6 in \cite{STVH}. If $\theta = 1$ (random walk case) and the innovation law is symmetric and continuous, then the Sparre Andersen formula \eqref{Sparre} provides
\begin{equation*}
   p_{n}(1)\; =\; \frac{1}{4^{n}}\,\binom{2n}{n}, 
\end{equation*}
and thus $\Prob [\wtil{T}_{\theta} = n] = p_{n}(1)-p_{n+1}(1)=C_{n}/2^{2n+1}$, where $\{C_{n},\, n\ge 0\}$ denotes the sequence of Catalan numbers. Their classical integral representation leads to
\begin{equation*}
   \Prob [\wtil{T}_{\theta} = n]\; =\; \frac{1}{\pi}\int_{0}^1 x^{n}\, \sqrt{\frac{1-x}{x}}\, dx 
\end{equation*} 
for all $n\ge 0$ and shows that $\{\Prob [\wtil{T}_{\theta} = n], \, n\ge 0\}$ is a sequence of positive moments and therefore log-convex. It follows by the Goldie-Steutel criterion, see e.g.~Theorem II.10.1 in \cite{STVH}, that $\wtil{T}_{\theta}$ is also ID.

\vspace{.2cm}
In the true Markovian situation $\theta\in(0,\frac{1}{2}]$ with uniform innovation law on $[-1,1]$, the moment sequence argument still applies as established by the following result.

\begin{Prop}\label{ID1}
The random variable $\wtil{T}_{\theta}^{\,U}$ is {\rm ID} for any $\theta \in[0,\frac{1}{2}].$
\end{Prop}

\begin{proof}
For $\theta$ as stated, we infer from \eqref{eq:MR-id} and \eqref{SerRep} that 
\begin{equation*}
   p_{n}^{U}(\theta) \; =\; \sum_{k\ge 1}\; \frac{1}{a_{k}(\theta)\lambda_{k}(\theta)^{n}}, 
\end{equation*}
where $\lambda_{k}(\theta) = 2(1-\theta)a_{k}(\theta)$ is an increasing sequence with $\lambda_{1}(\theta) > 1$. This yields
\begin{equation*}
   \Prob [\wT_{\theta}^{\,U} = n]\; =\;p_{n}^{U}(\theta)\, -\, p_{n+1}^{\,U}(\theta)\; =\; \sum_{k\ge 1}\; \frac{\lambda_{k}(\theta) -1}{a_{k}(\theta)\, (\lambda_{k}(\theta))^{n+1}}\; =\; \int_{0}^1\! x^{n}\,\nu_{\theta}(dx), 
\end{equation*}
with
\begin{equation*}
   \nu^{}_{\theta}(dx)\; =\; \sum_{k\ge 1}\; \lpa\frac{\lambda_{k}(\theta)-1}{a_{k}(\theta)\, \lambda_{k}(\theta)}\rpa \delta_{1/\lambda_{k}(\theta)}(dx), 
\end{equation*}
and since $\nu_{\theta}$ is a positive measure on $(0,1)$, we arrive at the desired conclusion as in the case $\theta = 1$.\qed
\end{proof}

\begin{Rem}\rm
(a) The moment sequence representation argument to show log-convexity does no longer work if $\theta =-1$. Indeed, by using Euler's summation for the cotangent, we have 
\begin{align*}
\sum_{n\ge 0}p_{n}^{U}(-1)z^{n}\, &=\, \frac{1 + \sin(z/2)}{\cos(z/2)}\,=\, \cot\bigg(\frac{\pi-z}{4}\bigg)\, =\, \lim_{n\to\infty}\sum_{\vert j\vert \le n}\frac{4}{\pi(4j+1) -z},
\end{align*}
which after some simple transformations leads to
\begin{equation*}
   \Prob[\wT_{\theta}^{\,U} = n]\; =\; \int_{-1}^1 x^{n}\,\nu(dx) 
\end{equation*}
with $\lambda_j = 1/(\pi(4j+1))$ for all $j\in\Z$ and
\begin{equation*}
   \nu(dx)\,=\,  4(1-x) \sum_{j\in\Z}\lambda_j\,\delta_{\lambda_j} (dx). 
\end{equation*}
But the latter is obviously only a signed measure on $(-1,1)$. 

\vspace{.1cm}
(b) If $\theta\ge 2$ and thus $\Prob[\wT_{\theta}^{\,U}=\infty]>0$ (defective case), the ID of $\wT_{\theta}^{\,U}$ remains an open problem. Yet, we remark that
\begin{equation*}
   \wJ_{n+1}(1/\theta)\,\wJ_{n-1}(1/\theta)\;\ge \; (1+1/n)\,\wJ_{n}(1/\theta)^{2}\quad\text{for all }n\ge 2 
\end{equation*}
provides a sufficient condition for the log-convexity of $\{\Prob[\wT_{\theta}^{\,U}=n],\,n\ge 1\}$ and thus the ID of $\wT_{\theta}^{\,U}$, as can be shown with the help of Proposition \ref{Hit}. Another line of attack might be the Wiener-Hopf type factorization
\begin{equation*}
    \lpa 1-\Erw\lcr z^{\wT^{\,U}_{1/\theta}}\rcr\rpa\lpa 1-\Erw\lcr z^{\wT_{\theta}^{\,U}}\,\1_{\{\wT_{\theta}^{\,U} < \infty\}}\rcr\rpa\; =\; 1-z, 
\end{equation*}
valid for all $z\in (0,1)$ and a straightforward consequence of Theorem~\ref{thm:main_3}. In support of this, we recall that for L\'evy processes, the Wiener-Hopf factors are indeed ID random variables. Finally, we point out that, because of the negative signs, the recursive formula \eqref{Jrec} does not give, at least not directly, the canonical representation of a discrete ID distribution as stated in Theorem~II.4.4. of \cite{STVH}.

\vspace{.1cm}
(c) The argument given in the above proof of Proposition \ref{ID1} amounts to a total positivity property for Mallows-Riordan polynomials that is mentioned in \cite{Sokal14}. To explain, recall that
\begin{equation*}
   \frac{J_{n+1}(\theta)}{n!} \; =\; \sum_{k\ge 1}\; \frac{1}{a_{k}(\theta)b_{k}(\theta)^{n}}\ =\ \int_{0}^{\infty}\! x^{n}\,\wtil{\nu}_{\theta}(dx) 
\end{equation*}
for all $\theta \in (0,1)$, where the $b_{k}(\theta):=(1-\theta)a_{k}(\theta)=\lambda_{k}(\theta)/2$ are positive and increasing numbers, and 
\begin{equation*}
   \wtil{\nu}_{\theta}(dx)\ =\ \sum_{k\ge 1}\; \frac{1}{a_{k}(\theta)}\, \delta_{1/b_{k}(\theta)} (dx)
\end{equation*}
a positive measure on $(0,\infty)$. By Stieltjes' criterion, this entails the total positivity of the Hankel matrix
\begin{equation*}
   \lcr\frac{J_{i+j+1}(\theta)}{(i+j)!}\rcr_{i,j\ge 0} 
\end{equation*}
for any $\theta\in(0,1)$, which means that all minors of this matrix are non-negative. Now it has been conjectured in \cite{Sokal14} that this Hankel matrix is even coefficientwise totally positive, that is, all minors are polynomials with nonnegative coefficients. But even the assertion that the polynomial
\begin{equation*}
    J_{n+1}(\theta)\, J_{n-1}(\theta)\, -\, (1+1/n) \, (J_n(\theta))^2
\end{equation*}
has only nonnegative coefficients for all $n\ge 1,$ which provides to the coefficientwise log-convexity of the sequence $\{J_{n+1}(\theta)/n!,\, n\ge 0\}$, remains an open problem.
\end{Rem}

Propositions \ref{Asymp1}, \ref{Asymp2} and \ref{Hout} have shown that $\Erw[T_{\theta}^{\,U}] < \infty$ if and only if  $\theta < 1$. In this case, the random variable $\hT_{\theta}^{\,U}$ with law
\begin{equation*}
    \Prob[\hT_{\theta}^{\,U} = n]\; =\; \frac{\Prob[T_{\theta}^{\,U} > n]}{\Erw[T_{\theta}^{\,U}]} 
\end{equation*}
for all $n\ge 0$ can be considered. This law is called the size-biasing of the law of $T_{\theta}^{U}$ and appears, for instance, in renewal theory. Regarding infinite divisibility, we have the following result.

\begin{Prop}\label{ID2}
The random variable $\hT_{\theta}^{\,U}$ is {\rm ID} for $\theta \in[-1,\frac{1}{2}]$ and fails to be so for $\theta < -1.$
\end{Prop}

\begin{proof}
If $\theta\in [0,\frac{1}{2}]$, the result follows from the log-convexity of the sequence $\{\Prob[\wT_{\theta}^{U}=n],\,n\ge 1\}$ shown in the proof of Proposition \ref{ID1} combined with  \cite[Proposition II.10.7]{STVH}, which then ensures the same property for the corresponding size-biased probabilities $\Prob[\hT_{\theta}^{\,U}=n]$. 

If $\theta\in [-1,0)$, we use Theorem \ref{thm:main_1} and the exponential formula \eqref{id:Mathar}, giving
\begin{align*}
\sum_{n\ge 0} \Prob[\hT_{\theta}^{\,U} = n]\, z^{n}\; &=\; \frac{1}{\Erw[T_{\theta}^{\,U}]}\, \sum_{n\ge 0} J_{n+1}(\theta)\,\frac{z^{n}}{2^{n}\, n!}\\
&=\; \exp\Bigg[-\sum_{n\ge 1} \frac{J_{n}(\theta)\, (1+\cdots + \theta^{n-1})}{2^{n}\, n!}\, (1-z^{n})\Bigg]
\end{align*}
for all $z\in (0,1]$. Now the discrete L\'evy-Khintchine formula and the fact that $J_{n}(\theta)\, (1+\cdots + \theta^{n-1})\,\ge\, 0$ for all $n\ge 1$ and $\theta\in[-1,0)$ imply that $\hT_{\theta}^{\,U}$ is ID. 

Finally, if $\theta < -1$, we note as a direct consequence of \eqref{eq:main_theorem_3} and \eqref{expwJ} that 
\begin{align*}
\sum_{n\ge 0}\, \Prob[\hT_{\theta}^{\,U}& = n]\, z^{n}  \\ & =\,\exp\Bigg[-\sum_{n\ge 1} \frac{(-1)^{n-1}J_{n}(1/\theta)\, (1+\cdots + \theta^{1-n})}{2^{n}\,n!}\, (1-z^{n})\Bigg],
\end{align*}
and this shows that $\hT_{\theta}^{\,U}$ cannot be ID because $(-1)^{n-1} J_{n}(1/\theta)\, (1+\cdots + \theta^{1-n})$ takes negative values.\qed
\end{proof}

\begin{Rem}\rm
(a) If $\theta < -1$ the argument just given has shown that the sequence $\{ \Prob[\hT_{\theta}^{\,U} = n],\, n\ge 0\}$ is not log-convex and so, again by \cite[Proposition II.10.7]{STVH}, that the sequence $\{ \Prob[T_{\theta}^{\,U} = n],\, n\ge 0\}$ is not log-convex either.\\

\vspace{.1cm}
(b) If $\theta =-1$, one has $J_{n}(-1)\, (1+\cdots + (-1)^{1-n}) = B_{n-1}\ge 0$, where $B_{n}$ is Euler's $n$-th secant or ``zig'' number, and so we retrieve the well-known formula
\begin{equation*}
   \sum_{n\ge 0}\, A_{n}\, \frac{z^{n}}{n!} \; =\;  \exp\Bigg[\, \sum_{n\ge 1}B_{n-1}\,\frac{z^{n}}{n!}\Bigg] \; =\;  \exp\Bigg[ \, \sum_{n\ge 0} A_{2n}\,\frac{z^{2n+1}}{(2n+1)!}\Bigg], 
\end{equation*}
which (see e.g.\ (1.2) in \cite{Stanley10}) amounts to
$$\int_0^x \frac{dt}{\cos t} \; =\; \log \lpa \frac{1 + \sin x}{\cos x}\rpa, \qquad x \in (0,\pi/2).$$
\end{Rem}

\subsection{Uniform innovations on non-symmetric intervals}
\label{NonSym}

Finally, we want to briefly discuss how some of our results can be extended to the case when the innovation law is uniform on $[-a,b]$ for arbitrary $a,b > 0$. Let $p_{n}^{a,b}(\theta)$ denote the corresponding persistence probability.

\begin{Prop}\label{Last}
For every $a,b > 0$ and $n\ge 0$, one has
\begin{align*}
p_{n}^{a,b}(\theta)\; =\; 
\begin{cases}
\displaystyle\lpa\frac{b}{a+b}\rpa^{\! n}\frac{J_{n+1}(\theta)}{n!}&\displaystyle\text{if }\theta\in\left[-1,\frac{a}{a+b}\right],\\[3mm]
\displaystyle\lpa\frac{b}{a+b}\rpa^{\! n}\frac{\wJ_{n+1}(1/\theta)}{n!}&\text{if }\theta\in(-\infty,-1].
\end{cases}
\end{align*}
\end{Prop}

\begin{proof}
Given uniform innovations on $[-a,b]$ in \eqref{eq:AR(1)}, it is no loss of generality to assume $b=1$, for otherwise this holds for the innovations $\tX_{n} = b^{-1}X_{n}$ upon multiplication of \eqref{eq:AR(1)} by $b^{-1}$ and $p_{n}^{a,b}(\theta) = \Prob[\tY_{1} > 0,\ldots, \tY_{n} > 0]$.
Recalling the discussion prior to \eqref{eq:why_theta<=1/2}, it is easy to see that in the present situation
\begin{align*}
p_{n}^{a,1}(\theta)\ &=\ \frac{1}{(a+1)^{n}}\int_{0}^{1}\!\int_{-\theta u_{1}}^{1}\!\!\!\cdots\int_{-(\theta u_{n-1}+\cdots+\theta^{n-1}u_{1})}^{1}du_{n}\ldots du_{2}\,du_{1}\\
&=\ \frac{J_{n+1}(\theta)}{(a+1)^{n}\,n!}
\end{align*}
for any $\theta \ge -1$ such that $\theta +\cdots + \theta^{n-1}\le a$. But the latter holds for all $n\ge 0$ if $\theta\in [-1,a/(a+1)]$. If $\theta<-1$, we observe that
\begin{align*}
p_{n}^{a,1}(\theta)\ &=\ \frac{1}{(a+1)^{n}}\int_{0}^{1}\!\int_{(-\theta u_{1})\wedge 1}^{1}\!\!\!\cdots\int_{(-\theta u_{n-1}-\cdots-\theta^{n-1}u_{1})\wedge 1}^{1}du_{n}\ldots du_{2}\,du_{1}
\end{align*}
because each of the lower integration bounds are $\ge 0$ and thus independent of $a$. In other words, $a\mapsto (a+1)^{n} p_{n}^{a,1}(\theta)$ is constant for each $n\ge 1$, and the constant equals $2^{n}p_{n}^{1,1}(\theta)=\wJ_{n+1}(1/\theta)/n!$ by Corollary \ref{cor:Thm 2}.\qed
\end{proof}

\begin{Rem}\label{LastR}\rm
(a) Proposition \ref{Last} covers the comfortable cases when $(a+1)^{n}$ $p_{n}^{a,1}(\theta)=2^{n}p_{n}^{1,1}(\theta)$ so that we can give closed-form expressions by resorting to our results for symmetric uniform innovations. For $\theta \ge 1+1/a$, this comfortable situation does no longer occur whence a closed-form expression cannot be derived from Theorem \ref{thm:main_3} and its corollary. In the random walk case $\theta = 1$, the same disclaimer applies whenever $a\ne 1$. Finally, as a consequence of the previous result combined with what has been pointed out in Remarks~\ref{Inc1} and \ref{Inc3}, we note that $\theta\mapsto p_{n}^{a,b}(\theta)$ is non-decreasing on $\R$ for all $a,b>0$ and $n\ge 0$.

\vspace{.1cm}
(b) Regarding asymptotic behavior, it follows from Propositions \ref{Asymp1} and \ref{Asymp2} that
\begin{align*}
p_{n}^{a,b}(\theta)\; \sim\; 
\begin{cases}
\displaystyle\frac{1}{z_{\theta}} \lpa\frac{b}{(a+b)(1-\theta)\, z_{\theta}}\rpa^{n}&\displaystyle\text{if }\theta\in\left[-1,\frac{a}{a+b}\right],\\[3mm]
\displaystyle\frac{1}{c_{\theta}} \lpa\frac{b}{(a+b)(1-\theta)\, z_{1/\theta}}\rpa^{\! n}&\text{if }\theta\in(-\infty,-1],
\end{cases}
\end{align*}
as $n\to\infty$, where $z_{\theta} = \inf\{z > 0:\; E(\theta,-z) =0\}$ and $c_{\theta}$ denotes a positive constant. The first asymptotics extends Proposition 3.1 in \cite{AMZ21} dealing with the zigzag case $\theta = -1$. From Proposition \ref{ID1}, we can also deduce that the shifted first passage time $\wT^{\,a,b}_{\theta}$, with obvious meaning, is ID for all $\theta\in [0, a/(a+b)]$.

\vspace{.1cm}
(c) For any $a,b >0$ and $\theta\in\R$, the truncated Volterra endomorphism $K$ on $\cC_b(\R^+,\R)$, defined by
\begin{equation*}
    K\psi(x)\; =\; \frac{1}{a+b}\, \int_{-a}^b \psi (y + \theta x) \, \1_{\{y+\theta x \ge 0\}}\, dy,\qquad x \ge 0 
\end{equation*}
is totally bounded and equicontinuous and hence compact by the Arzel\`a-Ascoli theorem. It follows from Theorem 2.1. in \cite{AMZ21} and the above asymptotics that its largest eigenvalue equals
\begin{gather*}
\frac{b}{(a+b)(1-\theta)\,z_{\theta}}\, <\, 1\qquad\text{if }\theta\in\left[-1,\frac{a}{a+b}\right]
\shortintertext{and}
\frac{b}{(a+b)(1-\theta)\, z_{1/\theta}}\, <\, 1\qquad\text{if }\theta\in \theta\in(-\infty,-1].
\end{gather*}
If $\theta = 0$, the largest eigenvalue equals $b/(a+b)$ and the corresponding eigenvectors are the constant functions. If $\theta = -1$, the largest eigenvalue is $2\pi b/(a+b)$ and the corresponding eigenvectors are the constant multiples of $\cos(\pi x/2b)\1_{[0,b]}(x)$. In all other cases, the eigenvectors are the solutions to certain delayed ODE's and of unknown  explicit form. Finally, it would also be interesting to know if the largest eigenvalue of this truncated Volterra operator is computable in the case $\theta = (a/(a+b), 1).$
\end{Rem}

\subsection*{Acknowledgments} 
Part of this work has been written during a stay at the Technical
University of Berlin of the fourth author, who would like to thank
Jean-Dominique Deuschel for the very good working conditions.

\bibliographystyle{siam}

\end{document}